\documentclass[11pt, reqno]{amsart}
\usepackage{amsmath, amsthm, amsfonts, amssymb, amsbsy, bigstrut, graphicx, enumerate,  upref, longtable, comment, booktabs, array, caption, subcaption, mathdots, mathtools}

\allowdisplaybreaks

\usepackage{pstricks}
\usepackage{times}

\usepackage[numbers, sort&compress]{natbib}

\usepackage[breaklinks]{hyperref}

\newcommand{\mm}{\mathcal{M}}

\newcommand{\mc}{\mathcal{C}}

\newtheorem{thm}{Theorem}[section]
\newtheorem{lmm}[thm]{Lemma}

\newtheorem{prop}[thm]{Proposition}

\theoremstyle{definition}
\newtheorem{remark}[thm]{Remark}
\newtheorem{ex}[thm]{Example}

\newcommand{\ee}{\mathbb{E}}

\newcommand{\mf}{\mathcal{F}}

\newcommand{\pp}{\mathbb{P}}

\newcommand{\var}{\mathrm{Var}}
\newcommand{\ve}{\varepsilon}

\numberwithin{equation}{section}

\usepackage[utf8]{inputenc}
\usepackage{amsmath}
\usepackage{amsfonts}
\usepackage{bbm}
\usepackage{mathtools}
\usepackage{tikz}
\usetikzlibrary{decorations.pathreplacing}
\usepackage{amsthm}
\usepackage[english]{babel}
\usepackage{fancyhdr}
\usepackage{amssymb}
\usepackage{enumitem}
\usepackage{qtree}
\usepackage{mathrsfs}

\newcommand{\N}{\mathbb{N}}

\renewcommand{\bar}{\overline}

\newcommand{\ER}{Erd\H{o}s--R\'enyi }

\usepackage{bbm}
\usepackage{stmaryrd}
\usepackage{latexsym}
\usepackage{amsfonts}
\usepackage{amsmath,amssymb,amscd}
\usepackage{yfonts}
\usepackage{dsfont}
\usepackage{mathabx}

 \newcommand{\Ln}{L^{(n)}}
 \newcommand{\Vn}{V^{(n)}}

\DeclareFontFamily{U}{txsyc}{}
\DeclareFontShape{U}{txsyc}{m}{n}{
   <-> txsyc%
}{}
\DeclareFontShape{U}{txsyc}{bx}{n}{
   <-> txbsyc%
}{}
\DeclareFontShape{U}{txsyc}{l}{n}{<->ssub * txsyc/m/n}{}
\DeclareFontShape{U}{txsyc}{b}{n}{<->ssub * txsyc/bx/n}{}
\DeclareSymbolFont{symbolsC}{U}{txsyc}{m}{n}
\SetSymbolFont{symbolsC}{bold}{U}{txsyc}{bx}{n}
\DeclareFontSubstitution{U}{txsyc}{m}{n}
\DeclareMathSymbol{\df}{\mathrel}{symbolsC}{"42}
\DeclareMathSymbol{\fd}{\mathrel}{symbolsC}{"43}
\DeclareMathSymbol{\lJoin}{\mathrel}{symbolsC}{"58}
\DeclareMathSymbol{\rJoin}{\mathrel}{symbolsC}{"59}

\newcommand{\f}[2]{\frac{#1}{#2}}
\newcommand{\cA}{\mathcal{A}}
\newcommand{\cB}{\mathcal{B}}

\newcommand{\cE}{\mathcal{E}}
\newcommand{\cF}{\mathcal{F}}
\newcommand{\cG}{\mathcal{G}}

\newcommand{\cI}{\mathcal{I}}

\newcommand{\cN}{\mathcal{N}}

\newcommand{\cS}{\mathcal{S}}
\newcommand{\cT}{\mathcal{T}}

\newcommand{\NN}{\mathbb{N}}
\newcommand{\PP}{\mathbb{P}}

\newcommand{\RR}{\mathbb{R}}
\renewcommand{\SS}{\mathbb{S}}
\newcommand{\TT}{\mathbb{T}}

\newcommand{\ZZ}{\mathbb{Z}}

\newcommand{\iy}{\infty}
\newcommand{\lt}{\left}
\newcommand{\me}{\medskip}

\newcommand{\pa}{\partial}
\newcommand{\ri}{\rightarrow}
\newcommand{\rt}{\right}

\newcommand{\wi}{\widetilde}

\newcommand{\fo}{\forall\ }

\newcommand{\lVe}{\lt\Vert}

\newcommand{\rVe}{\rt\Vert}
\newcommand{\st}{\,:\,}

\newcommand{\un}{\mathds{1}}

\newcommand{\bq}{\begin{eqnarray*}}
\newcommand{\bqn}[1]{\begin{eqnarray}\label{#1}}
\newcommand{\eq}{\end{eqnarray*}}
\newcommand{\eqn}{\end{eqnarray}}

\newcommand{\lin}{\llbracket}
\newcommand{\rin}{\rrbracket}

\newcommand{\ttsim}{\raise.17ex\hbox{$\scriptstyle\mathtt{\sim}$}}
\newcommand{\kh}{\kern .08em}

\begin{document}
\title{A random walk on the Rado graph}
\author{Sourav Chatterjee}
\address{Departments of mathematics and statistics, Stanford University}
\email{souravc@stanford.edu}
\author{Persi Diaconis}
\address{Departments of mathematics and statistics, Stanford University}
\email{diaconis@math.stanford.edu}
\author{Laurent Miclo}
\address{Economics Department, CNRS, University of Toulouse}
\email{laurent.a.miclo@gmail.com}
\thanks{S.~C.'s research was partially supported by NSF grants DMS-1855484 and DMS-2113242}
\thanks{P.~D.'s research was partially suppported by NSF grant DMS-1954042}
\thanks{L.~M.~acknowledges funding from ANR grant ANR-17-EUR-0010.}

\dedicatory{Dedicated to our friend and coauthor Harold Widom.}
\subjclass[2020]{}

\begin{abstract}
The Rado graph, also known as the random graph $G(\infty, p)$, is a classical limit object for finite graphs. We study natural ball walks as a way of understanding the geometry of this graph. For the walk started at $i$, we show that order $\log_2^*i$ steps are sufficient, and for infinitely many $i$, necessary for convergence to stationarity. The proof involves an application of Hardy's inequality for trees.
\end{abstract}

\maketitle

\section{Introduction}
The Rado graph $R$ is a natural limit of the set of all finite graphs (Fraiss\'e limit, see Section \ref{fraisse}). In Rado's construction, the vertex set is $\N = \{0,1,2,\ldots\}$. There is an undirected edge from $i$ to $j$ if $i<j$ and the $i^{\textup{th}}$ binary digit of $j$ is a one (where the $0^{\textup{th}}$ digit is the first digit from the right). Thus, $0$ is connected to all odd numbers, $1$ is connected to $0$ and all $j$ which are $2$ or $3$ (mod $4$) and so on. There are many alternative constructions. For $p\in (0,1)$, connecting $i$ and $j$ with probability $p$ gives the \ER graph $G(\infty, p)$, which is (almost surely) isomorphic to $R$. Further constructions are in Section \ref{fraisse}. 

Let $(Q(j))_{0\le j<\infty}$ be a positive probability on $\N$ (so, $Q(j)>0$ for all $j$, and $\sum_{j=0}^\infty Q(j)=1$). We study a `ball walk' on $R$ generated by $Q$: 
\begin{quote}
From $i\in \N$, pick $j\in N(i)$ with probability proportional to $Q(j)$, where $N(i) = \{j: j\sim i\}$ is the set of neighbors of $i$ in $R$.
\end{quote}
Thus, the probability of moving from $i$ to $j$ in one step is 
\begin{align}\label{kform}
K(i,j) = 
\begin{cases}
Q(j)/Q(N(i)) &\text{ if } i\sim j,\\
0 &\text{ otherwise.}
\end{cases}
\end{align}
As explained below, this walk is connected, aperiodic and reversible, with stationary distribution
\begin{align}\label{piform}
\pi(i) = \frac{Q(i)Q(N(i))}{Z},
\end{align}
where $Z$ is the normalizing constant. 

It is natural to study the mixing time --- the rate of convergence to stationarity. The following result shows that convergence is extremely rapid. Starting at $i\in \N$, order $\log_2^* i$ steps suffice, and for infinitely many $i$, are needed. 

\begin{thm}\label{mainthm}
Let $Q(j) = 2^{-(j+1)}$, $0\le j<\infty$. For $K(i,j)$ and $\pi$ defined at \eqref{kform} and \eqref{piform} on the Rado graph $R$, 
\begin{enumerate}
\item for universal $A,B>0$, 
\[
\|K_i^\ell - \pi\| \le A e^{\log_2^*i} e^{-B\ell} 
\]
for all $i\in \N$, $\ell \ge 1$, and 
\item for universal $C>0$, if $2^{(k)} = 2^{2^{\cdot^{\cdot^{\cdot^{2}}}}}$ is the tower of $2$'s of height $k$, 
\[
\|K^\ell_{2^{(k)}} - \pi\| \ge C
\]
for all $\ell\le k$. Here $\|K_i^\ell - \pi\| = \frac{1}{2}\sum_{j=0}^\infty |K^\ell(i,j) - \pi(j)|$ is the total variation  distance and $\log_2^* i $ is the number of times $\log_2$ needs to be applied, starting from $i$, to get a result $\le 1$. 
\end{enumerate}
\end{thm}
The proofs allow for some variation in the measure $Q$. They also work for the $G(\infty, p)$ model of $R$,  though some modification is needed since then $K$ and $\pi$ are random. 

Theorem \ref{mainthm} answers a question in Diaconis and Malliaris~\cite{diaconis2021complexity}, who proved the lower bound. Most Markov chains on countable graphs restrict attention to locally finite graphs~\cite{woess}. For Cayley graphs, Bendikov and Saloff-Coste~\cite{bendikovsaloffcoste} begin the study of more general transitions and point out how few tools are available. See also \cite{gh20, martineau}. Studying the geometry of a space (here $R$) by studying the properties of the Laplacian (here $I-K$) is a classical pursuit (``Can you hear the shape of a drum?'') --- see \cite{jorgensonlang}. 

Section \ref{background} gives background on the Rado graph, Markov chains and ball walks, and Hardy's inequalities. Section \ref{preliminaries} gives preliminaries on the behavior of the neighborhoods of the $G(\infty, p)$ model. The lower bound in Theorem \ref{mainthm} is proved in Section \ref{lowerproof}. Both Sections \ref{preliminaries} and \ref{lowerproof} give insight into the geometry of $R$. The upper bound on Theorem \ref{mainthm} is proved by proving that the Markov chain $K$ has a spectral gap. Usually, a spectral gap alone does not give sharp rates of convergence. Here, for any start $i$, we show the chain is in a neighborhood of $0$ after order $\log_2^* i$ steps. Then the spectral gap shows convergence in a bounded number of further steps. This argument works for both models of $R$. It is given in Section \ref{upperproof}.

The spectral gap for the $G(\infty, p)$ model is proved in Section \ref{spectralgap} using a version of Cheeger's inequality for trees. For Rado's binary model, the spectral gap is proved by a novel version of Hardy's inequality for trees in Section \ref{spectralgap2}. This is the first probabilistic application of this technique, which we hope will be useful more generally. 
There are two appendices containing technical details for the needed versions of Cheeger's and Hardy's inequalities. 

\vskip.2in
\noindent{\bf Acknowledgments:} We thank Peter Cameron, Maryanthe Malliaris, Sebastien Martineau, Yuval Peres, and Laurent Saloff-Coste for their help.

\section{Background on $R$, Markov chains, and Hardy's inequalities}\label{background}

\subsection{The Rado graph}\label{fraisse}
A definitive survey on the Rado graph (with full proofs) is in Peter Cameron's fine article \cite{cameron}.  We have also found the Wikipedia entry on the Rado graph and Cameron's follow-up paper \cite{cameron2} useful.

In Rado's model, the graph $R$ has vertex set $\N = \{0,1,2,\ldots\}$ and an undirected edge from $i$ to $j$ if $i<j$ and the $i^{\textup{th}}$ digit of $j$ is a one. There are many other constructions. The vertex set can be taken as the prime numbers that are $1$ (mod $4$) with an edge from $p$ to $q$ if the Legendre symbol $(\frac{p}{q}) = 1$. In \cite{diaconis2021complexity}, the graph appears as an induced subgraph of the commuting graph of the group $U(\infty, q)$ --- infinite upper-triangular matrices with ones on the diagonal and entries in $\mathbb{F}_q$. The vertices are points of $U(\infty, q)$. There is an edge from $x$ to $y$ if and only if the commutator $x^{-1} y^{-1}xy$ is zero. The infinite \ER graphs $G(\infty,p)$ are almost surely isomorphic to $R$ for all $p$, $0<p<1$. 

The graph $R$ has a host of fascinating properties:
\begin{itemize}
\item It is stable in the sense that deleting any finite number of vertices or edges yields an isomorphic graph. So does taking the complement.
\item It contains all finite or countable graphs as induced subgraphs. Thus, the (countable) empty graph and complete graphs both appear as induced subgraphs. 
\item The diameter of $R$ is two --- consider any $i\ne j \in \N$ and let $k$ be a binary number with ones in positions $i$ and $j$ and zero elsewhere. Then $i\sim k\sim j$. 
\item Each vertex is connected to ``half'' of the other vertices: $0$ is connected to all the odd vertices, $1$ to $0$ and all numbers congruent to $2$ or $3$ (mod $4$), and so on. 
\item $R$ is highly symmetric: Any automorphism between two induced subgraphs can be extended to all of $R$ (this is called homogeneity). The automorphism group has the cardinality of the continuum. 
\item $R$ is the ``limit'' if the collection of all finite graphs (Fraiss\'e limit). Let us spell this out. A {\it relational structure} is a set with a finite collection of relations (we are working in first order logic without constants or functions). For example, $\mathbb{Q}$ with $x<y$ is a relational structure. A graph is a set with one symmetric relation. The idea of a ``relational sub-structure'' clearly makes sense. A class $\mc$ of structures has the {\it amalgamation property} if for any $A, B_1, B_2\in \mc$ with embeddings $A\stackrel{f_1}{\to} B_1$ and $A\stackrel{f_2}{\to} B_2$, there exists $C\in \mc$ and embeddings $B_1\stackrel{g_1}{\to} C$ and $B_2 \stackrel{g_2}{\to} C$ such that $g_1f_1 = g_2f_2$. A countable relational structure $M$ is {\it homogeneous} if 
any isomorphism between finite substructures can be extended to an 
automorphism of $M$. Graphs and $\mathbb{Q}$ are homogeneous relational structures. A class $\mc$ has the joint embedding property if for any $A,B \in \mc$ there is a $C \in \mc$ so that $A$ and $B$ are embeddable in $C$.
\begin{thm}[Fraiss\'e]
Let $\mc$ be a countable class of finite structures with the joint embedding property and closed under `induced' isomorphism with amalgamation. Then there exists a unique countable homogeneous $\mm$ with $\mc$ as induced substructures. 
\end{thm}
The rationals $\mathbb{Q}$ are the Fraiss\'e limit of finite ordered sets. The Rado graph $R$ is the Fraiss\'e limit of finite graphs. We have (several times!) been told ``for a model theorist, the Rado graph is just as interesting as the rationals''. 
\end{itemize}
There are many further, fascinating properties of $R$; see \cite{cameron}.
\subsection{Markov chains}
A {\it transition matrix} $K(i,j)$, $0\le i,j<\infty$, $K(i,j)\ge 0$, $\sum_{j=0}^\infty K(i,j)=1$ for all $i$, $0\le i<\infty$, generates a Markov chain through its powers
\[
K^\ell(i,j) = \sum_{k=0}^\infty K(i,k) K^{\ell-1}(k,j).
\]
A probability distribution $\pi(i)$, $0\le i<\infty$, is {\it reversible} for $K$ if 
\begin{align}\label{revcond}
\pi(i) K(i,j) = \pi(j) K(j,i) \ \ \text{ for all } 0\le i,j<\infty.
\end{align}
\begin{ex}
With definitions \eqref{kform}, \eqref{piform} on the Rado graph, if $i\sim j$,
\[
\pi(i)K(i,j) = \frac{Q(i)Q(N(i))}{Z} \frac{Q(j)}{Q(N(i))} = \frac{Q(i)Q(j)}{Z} = \pi(j) K(j,i).
\]
(Both sides are zero if $i\not\sim j$.) 
\end{ex}
In the above example, we think of $K(i,j)$ as a `ball walk': From $i$, pick a neighbor $j$ with probability proportional to $Q(j)$ and move to $j$. We initially found the neat reversible measure surprising. Indeed, we and a generation of others thought that ball walks  would have $Q$ as a stationary distribution. Yuval Peres points out that, given a probability $Q(j)$ on the vertices, assigning symmetric weight $Q(i)Q(j)$ to $i\sim j$ gives this $K$ for the weighted local walk. A double ball walk --- ``from $i$, choose a neighbor $j$ with probability proportional to $Q(j)$, and from $j$, choose a neighbor $k$ with probability proportional to $Q(k)/Q(N(k))$'' --- results in a reversible Markov chain with $Q$ as reversing measure. Note that these double ball walks don't require knowledge of normalizing constants. All of this suggests ball walks as reasonable objects to study.

Reversibility \eqref{revcond} shows that $\pi$ is a stationary distribution for $K$:
\[
\sum_{i=0}^\infty\pi(i) K(i,j) = \sum_{i=0}^\infty \pi(j)K(j,i) = \pi(j)\sum_{i=0}^\infty K(j,i) = \pi(j).
\]
In our setting, since the Rado graph has diameter $2$, the walk is connected. It is easy to see that it is aperiodic. Thus, the $\pi$ in \eqref{piform} is the unique stationary distribution. Now, the fundamental theorem of Markov chain theory shows, for every starting state $i$, $K^\ell(i,j) \to \pi(j)$ as $\ell \to \infty$, and indeed,
\[
\lim_{\ell \to\infty} \|K^\ell_i - \pi\|=0.
\]
Reversible Markov chains have real spectrum. Say that $(K,\pi)$ has a {\it spectral gap} if there is $A>0$ such that for every $f\in \ell^2(\pi)$, 
\begin{align}\label{spectral}
\sum_i (f(i)-\bar{f})^2 \pi(i) &\le A \sum_{i,j} (f(i)-f(j))^2 \pi(i) K(i,j),
\end{align}
where $\bar{f} = \sum_{i=0}^\infty f(i)\pi(i)$. (Then the gap is at least $1/A$.) For chains with a spectral gap, for any $i$,
\begin{align}\label{spectral2}
4\|K_i^\ell - \pi\|^2 \le \frac{1}{\pi(i)}\biggl(1-\frac{1}{A}\biggr)^{2\ell}. 
\end{align}
Background on Markov chains, particularly rates of convergence, can be found in the readable book of Levin and Peres~\cite{levinperes}. For the analytic part of the theory, particularly \eqref{spectral} and \eqref{spectral2}, and many refinements, we recommend \cite{MR1490046}.

There has been a healthy development in Markov chain circles around the theme `How does a Markov chain on a random graph behave?'. One motivation being, `What does a typical convergence rate look like?'. The graphs can be restricted in various natural ways (Cayley graphs, regular graphs of fixed degree or fixed average degree, etc.). A survey of by now classical work is Hildebrand's survey of `random-random walks' \cite{hildebrand}. Recent work by Bordenave and coauthors can be found from \cite{bordenave1, bordenave2}. For sparse \ER graphs, there is remarkable work on the walk restricted to the giant component. See \cite{nachmiasperes}, \cite{fountoulakisreed} and \cite{berestyckilubetzkyperessly}.

It is worth contrasting these works with the present efforts. The above results pick a neighbor uniformly at random. In the present paper, the ball walk drives the walk back towards zero. The papers above are all on finite graphs. The Makov chain of Theorem \ref{mainthm} makes perfect sense on finite graphs. The statements and proofs go through (with small changes) to show that order $\log_2^* i$ steps are necessary and sufficient.  (For the uniform walk on $G(n,1/2)$, a bounded number of steps suffice from most initial states, but there are states from which $\log_2^*n $ steps are needed.)

\subsection{Hardy's inequalities}
A key part of the proof of Theorem \ref{mainthm} applies Hardy's inequalities for trees to prove a Poincar\'e inequality (Cf.~\eqref{spectral}) and hence a bound on the spectral gap. Despite a large expository literature, Hardy's inequalities remain little known among probabilists. Our application can be read without this expository section but we hope that some readers find it useful. Extensive further references, trying to bridge the gap between probabilists and analysts, is in \cite{klaassenwellner}.

Start with a discrete form of Hardy's original inequality \cite[pp.~239--243]{hardy}. This says that if $a_n\ge 0$, $A_n = a_1+\cdots+a_n$, then
\[
\sum_{n=1}^\infty \frac{A_n^2}{n^2} \le 4\sum_{n=1}^\infty a_n^2,
\]
and the constant $4$ is sharp. Analysts say that ``the Hardy operator taking $\{a_n\}$ to $\{A_n/n\}$ is bounded from $\ell^2$ to $\ell^2$''. Later writers showed how to put weights in. If $\mu(n)$ and $\nu(n)$ are positive functions, one aims for
\[
\sum_{n=1}^\infty A_n^2 \mu(n) \le A \sum_{n=1}^\infty a_n^2\nu(n),
\]
for an explicit $A$ depending on $\mu(n)$ and $\nu(n)$. If $\mu(n)= 1/n^2$ and $\nu(n)=1$, this gives the original Hardy inequality. 

To make the transition to a probabilistic application, take $a(n) = g(n)-g(n-1)$ for $g$ in $\ell^2$. The inequality becomes
\begin{align}\label{hardy2}
\sum_{n=1}^\infty g(n)^2 \mu(n) \le A \sum_{n=1}^\infty(g(n)-g(n-1))^2 \nu(n).
\end{align}
Consider a `birth and death chain' which transits from $j$ to $j+1$ with probability $b(j)$ and from $j$ to $j-1$ with probability $d(j)$. Suppose that this has stationary distribution $\mu(j)$ and that $\sum_j g(j)\mu(j)=0$. Set $\nu(j) = \mu(j)d(j)$. Then \eqref{hardy2} becomes (following simple manipulations)
\begin{align}\label{hardy3}
\var(g) \le A \sum_{j,k} (g(j)-g(k))^2 \mu(j)K(j,k)
\end{align}
with $K(j,k)$ the transition matrix of the birth and death chain. This gives a Poincar\'e inequality and spectral gap estimate.

A crucial ingredient for applying this program is that the constant $A$ must be explicit and manageable. For birth-death chains, this is indeed the case. See \cite{MR1709277} or the applications in \cite{diaconismatchettwood}. 

The transition from \eqref{hardy2} to \eqref{hardy3} leans on the one-dimensional setup of birth-death chains. While there is work on Hardy's inequalities in higher dimensions, it is much more complex; in particular, useful forms of good constants $A$ seem out of reach. In \cite{MR1709277}, Miclo has shown that for a general Markov transition matrix $K(i,j)$, a spanning tree with the graph underlying $K$ can be found. There is a useful version of Hardy's inequality for trees due to Evans, Harris and Pick~\cite{MR1345719}. This is the approach developed in Section~\ref{spectralgap2} below which gives further background and details. 

Is approximation by trees good enough? There is some hope that the best tree is good enough (see \cite{zbMATH01054104}). In the present application, the tree chosen gives the needed result.

\subsection{The $\log^*$ function}
Take any $a>1$. The following is a careful definition of $\log_a^*x$ for $x\ge 0$. First, an easy verification shows that the map $x\mapsto (\log x)/x$ on $(0,\infty)$ is unimodal, with a unique maximum at $x = e$ (where its value is $1/e$), and decaying to $-\infty$ as $x\to 0$ and to $0$ as $x\to \infty$. Thus, if $a > e^{1/e}$, then for any $x>0$,
\[
\log_a x = \frac{\log x}{\log a} \le \frac{x}{e\log a} < x.
\]
Since $\log_a$ is a continuous map, this shows that if we start with any $x>0$, iterative applications of $\log_a$ will eventually lead to a point in $(0,a)$ (because there are no fixed points of $\log_a$ above that, by the above inequality), and then another application of $\log_a$ will yield a negative number. This allows us to define $\log_a^*x$ as the minimum number of applications of $\log_a$, starting from $x$, that gets us a nonpositive result. 

If $a \le e^{1/e}$, the situation is a bit more complicated. Here, $\log a \le 1/e$, which is the maximum value of the unimodal map $x\mapsto (\log x)/x$. This implies that there exist exactly two points $0 < y_a\le x_a$ that are fixed points of $\log_a$ (with $y_a=x_a$ if $a = e^{1/e}$). Moreover, $\log_a x < x$ if $x\notin [y_a,x_a]$, and $\log_a x\ge x$ if $x\in [y_a,x_a]$. Thus, the previous definition does not work. Instead, we define $\log_a^*x$ to be the minimum number of applications of $\log_a$, starting from $x$, that leads us to a result $\le x_a$. In both cases, defining $\log_a^*0=0$ is consistent with the conventions. Note that $\log_a^* x\ge 0$ for all $x\ge 0$.

\section{The geometry of the random model}\label{preliminaries}
Throughout this section the graph is $G(\infty,1/2)$ --- an \ER graph on $\N=\{0,1,2,\ldots\}$ with probability $1/2$ for each possible edge. From here on, we will use the notation $\N_+$ to denote the set $\{1,2,\ldots\}$ of strictly positive integers. Let $Q(x)=2^{-(x+1)}$ for $x \in \N$. The transition matrix 
\[
K(x,y) =\frac{Q(y)}{Q(N(x))} \un_{\{y\in N(x)\}}
\]
and its stationary distribution $\pi(x) = Z^{-1} Q(x)Q(N(x))$ are thus random variables. Note that $N(x)$, the neighborhood of $x$, is random. The main result of this section shows that this graph, with vertices weighted by $Q(x)$, has its geometry controlled by a tree rooted at $0$. This tree will appear in both lower and upper bounds on the mixing time for the random model.
  
 To describe things, let $p(x)=\min N(x)$ ($p$ is for `parent', not to be confused with the edge probability $p$ in $G(\infty, p)$).  We need some preliminaries about the mapping $p$.
\begin{lmm}\label{lem1}
Let $\cB$ be the event that for all $x\in \N_+$, $p(x) <x$. Then we have that $\pp(\cB)\ge 1/4$.
\end{lmm}
\begin{proof}
Denote 
\bq
\bar E&\df& \{\{x,y\}: x\neq y\in \N\},\eq
and for
any $e\in\bar E$, consider 
\bq
B_{e}&=&\un_E(e),\eq
where $E$ is the set of  edges in $G(\infty, 1/2)$, so that  $(B_e)_{e\in \bar E}$ is a family of  independent Bernoulli variables of parameter $1/2$.

For $x\in \N_+$, define $A_x$ the event that $x$ is not linked in $\cG$ to a smaller vertex. Namely, we have formally
\bq
 A_x&\df&\bigcap_{y\in\lin 0, x-1\rin} \{B_{\{y,x\}}=0\},\eq
 where $\lin 0, x-1\rin := \{0,1,\ldots,x-1\}$. Note that the family $(A_x)_{x\in\N_+}$ is independent, and in particular, its events are pairwise independent. We are thus in position to apply 
Kounias--Hunter--Worsley bounds \cite{zbMATH03283808,zbMATH03544899,zbMATH03784109} (see also the survey \cite{enwiki:998349836}), to see that for any $n\in\N_+$,
\bq
\PP\biggl(\bigcup_{x\in\lin 1, n\rin} A_x\biggr)&\leq &\min\biggl\{\sum_{x\in\lin 1, n\rin} \PP(A_x)-\PP(A_1)\sum_{y\in\lin 2, n\rin}\PP(A_y), 1\biggr\},\eq
where we used that
\bq
\PP(A_1)\geq \PP(A_2)\geq \cdots \geq \PP(A_n),
\eq
which holds because 
\bq
\fo x\in\N_+,\qquad \PP(A_x)&=&\prod_{y\in\lin 0, x-1\rin} \PP(B_{\{y,x\}}) = \f1{2^x}.
\eq
We deduce that
\bq
\PP\biggl(\bigcup_{x\in\lin1, n\rin} A_x\biggr)&\leq &\min\biggl\{\sum_{x\in\lin1,n\rin} \f1{2^x}-\f12\sum_{y\in\lin 2, n\rin}\f1{2^y}, 1\biggr\}\\
&=&\f12+\f14-\f1{2^{n+1}}.\eq
Letting $n$ tends to infinity, we get
\bq
\PP\biggl(\bigcup_{x\in\N_+} A_x\biggr)&\leq &\f34.\eq
To conclude, note that
\bq
\cB^{\mathrm{c}}&=&\bigcup_{x\in\N_+} A_x.\eq
\end{proof}

\begin{remark}\label{remKHW}
Assume that instead of $1/2$, the edges of $\bar E$ belong to $E$ with probability $p\in(0,1)$ (still independently),
the corresponding notions receive $p$ in index.
The above computations show
\bq
\PP_p(\cB)&\geq &1- (2-3p +p^2)\wedge 1,\eq
so that
$\PP_p(\cB)$ goes to $1$ as $p$ goes to $1$, but this bounds provides no information for  $p\in(0,(3-\sqrt{5})/2]$.
\end{remark}
\par
In fact the above observation shows that the Kounias--Hunter--Worsley bound is not optimal, at least for small $p>0$. So let us give another computation of $\PP_p(\cB)$:
\begin{lmm}\label{lem3}
Consider the situation described in Remark \ref{remKHW}, with $p\in(0,1)$. We have
\bq
\PP_p(\cB)&=&\biggl(\sum_{n\in\N} p(n)(1-p)^n\biggr)^{-1}\eq
where $p(n)$ is the number of partitions of $n$. In particular $\PP(\cB)>0$ for all $p\in(0,1)$.
\end{lmm}
\begin{proof}
Indeed, we have
\bq
\cB&=&\bigcap_{x\in\N_+} A_x^{\mathrm{c}},\eq
so that by independence of the $A_x$, for $x\in\N_+$,
\bq
\PP_p(\cB)&=&\prod_{x\in\N_+} \PP(A_x^{\mathrm{c}})\\
&=&\biggl(\prod_{x\in\N_+} \f1{1-(1-p)^x}\biggr)^{-1}\\
&=&\biggl(\prod_{x\in\N_+} \sum_{n\in\N}(1-p)^{xn}\biggr)^{-1}\\
\eq
Let $\cN$ be the set of sequences of integers $(n_l)_{l\in\N_+}$ with all but finitely many elements equal to zero.
Applying the distributive law to the above expression, we have
\bqn{contra}
\nonumber\PP_p(\cB)&=&\biggl(\sum_{(n_l)_{l\in\N_+}\in\cN} \prod_{x\in\N_+}(1-p)^{xn_x}\biggr)^{-1}\\
&=&\biggl(\sum_{n\in\N} p(n)(1-p)^n\biggr)^{-1}\eqn
where $p(n)$ is the number of ways to write $n$ as $\sum_{x\in\N_+} xn_x$, with $(n_l)_{l\in\N_+}\in\cN$. 
\end{proof}

Consider the set of edges
\bq
F&\df& \{\{x,p(x)\}: x\in\N_+\}\eq
and the corresponding  graph $\cT\df (\N,F)$. Under $\cB$, it is clear that $\cT$ is a tree. But this is always true:
\begin{lmm}\label{lem4}
The graph  $\cT$ is a tree.
\end{lmm}
\begin{proof}
The argument is by contradiction. Assume that $\cT$ contains a cycle, say $(x_l)_{l\in\ZZ_n}$ with $n\geq 3$.
Let us direct the a priori unoriented edges $\{x_l,x_{l+1}\}$, for $l\in\ZZ_n$, by putting an arrow from $x_l$ to $x_{l+1}$ (respectively from $x_{l+1}$ to $x_l$) if $p(x_l)=x_{l+1}$ (resp.\ $p(x_{l+1})=x_{l}$). Note that we either have 
\bqn{order}
\fo l\in\ZZ_n,\qquad x_l&\ri&x_{l+1},\eqn
or
\bq
\fo l\in\ZZ_n,\qquad x_{l+1}&\ri&x_{l},\eq
because otherwise there would exist $l\in\ZZ_n$ with two arrows exiting from $x_l$, a contradiction.
Up to reindexing $(x_l)_{l\in\ZZ_n}$ as $(x_{-l})_{l\in\ZZ_n}$, we can assume that \eqref{order} holds.
\par
Fix some $l\in\ZZ_n$. Since $p(x_l)=x_{l+1}$, we have $x_l\in N(x_{l+1})$, so $x_{l+2}=p(x_{l+1})\leq x_l$. Due to the fact that $x_l\neq x_{l+2}$ (recall that $n\geq 3$), we get
$x_{l+2}<x_l$. Starting from $x_0$ and iterating this relation (in a minimal way, $n/2$ times if $n$ is even, or $n$ times if $n$ is odd), we obtain a contradiction: $x_0<x_0$.
Thus, $\cT$ must be a tree.
\end{proof}
Let us come back to the case where $p=1/2$. The following result gives an idea of how far $p(x)$ is from $x$, for $x\in\N_+$.
\begin{lmm}\label{lem5}
Almost surely, there exist only finitely many $x\in\N_+$ such that $p(x)> 2\log_2(1+x)$.
In particular, a.s.\ there exists a (random) finite $C\geq 2$ such that
\bq
\fo x\in\N_+,\qquad p(x)&\leq & C\log_2(1+x).\eq
\end{lmm}
\begin{proof}
The first assertion follows from the Borel--Cantelli lemma, as follows.
For any $x\in\N_+$, consider the event
\bq
A_x&\df& \{p(x)>2\log_2(1+x)\}.\eq
Denoting $\lfloor\cdot\rfloor$ the the integer part,
we compute 
\bq
\sum_{x\in\N_+} \PP(A_x)&=&\sum_{x\in\N_+} 
\PP(B_{\{0,x\}}=0, B_{\{1,x\}}=0, ..., B_{\{\lfloor 2\log_2(1+x)\rfloor,x\}}=0)\\
&=&\sum_{x\in\N_+} 
\f1{2^{1+\lfloor 2\log_2(1+x)\rfloor}}\\
&\leq &\sum_{x\in\N_+} \f1{(1+x)^2}\\
&<&+\iy.\eq
Having shown that a.s.\ there exists only a finite number of integers $x\in\N_+$ satisfying $p(x)> 2\log_2(1+x)$, denote these points as $x_1, ..., x_N$, with $N\in\NN$.
To get the second assertion, it is sufficient to take
\bq
C&\df&\max\biggl\{\f{p(x_l)}{\log(1+x_l)}\st l\in\lin 1,N\rin\biggr\},\eq
with the convention that $C\df2$ if $N=0$.
\end{proof}



\section{The lower bound}\label{lowerproof}
The lower bound in Theorem \ref{mainthm}, showing that order $\log_2^*i$ steps are necessary for infinitely many $i$ is proved in \cite{diaconis2021complexity} for the binary model of the Rado graph and we refer there for the proof. A different argument is needed for the $G(\infty,1/2)$ model. This section gives the details (see Theorem~\ref{xthm} below).

Let $\mu$ be the stationary distribution of our random walk on $G(\infty, 1/2)$ (with $Q(j)=2^{-(j+1)}$, as in Theorem \ref{mainthm}), given a realization of the graph. Note that $\mu$ is random. For each $x\in \mathbb{N}$, let $\tau_x$ be the mixing time of the walk starting from $x$, that is, the smallest $n$ such that the law of the walk at time $n$, starting from $x$, has total variation distance $\le 1/4$ from $\mu$.  Note that the $\tau_x$'s are also  random. 
\begin{thm}\label{xthm}
Let $\tau_x$ be as above. Then with probability one, 
\[
\limsup_{x\to\infty} \frac{\tau_x}{\log_{16}^* x} \ge 1.
\]
Equivalently, with probability one, given any $\ve >0$, $\tau_{x} \ge (1-\ve)\log_{16}^* x$ for infinitely many $x$.
\end{thm} 

We need the following lemma.
\begin{lmm}\label{xlmm}
With probability one, there is an infinite sequence $x_0< x_1<x_2<\cdots\in \mathbb{N}$ such that: 
\begin{enumerate}
\item For each $i$, $x_{i+1}$ is connected to $x_i$ by an edge, but not connected by an edge to any other number in $\{0,1,\ldots, 2x_i-1\}$.
\item For each $i$, $2^{3x_i}\le x_{i+1} \le 2^{3x_i+1}-1$. 
\end{enumerate}
\end{lmm}
\begin{proof}
Define a sequence $y_0,y_1, y_2,\ldots$ inductively as follows. Let $y_0$ be an arbitrary element of $\mathbb{N}$. For each $i$, let $y_{i+1}$ be the smallest element in $\{2^{3y_i}, 2^{3y_i}+1,\ldots,2^{3y_i+1}-1\}$ that has an edge to $y_i$, but to no other number in $\{0,1,\ldots,2y_i-1\}$. If there exists no such number, then the process stops. Let $A_i$ be the event that $y_{i}$ exists. Note that $A_0\supseteq A_1\supseteq A_2\supseteq \cdots$.

Let $F(x) := 2^{3x}$ and $G(x):= 2^{3x+1}-1$. Let $a_0=b_0=y_0$, and for each $i\ge 1$, let 
\[
a_i := \underbrace{F\circ F\circ \cdots \circ F}_{\text{$i$ times}} (y_0), \ \ \ b_i := \underbrace{G\circ G\circ \cdots \circ G}_{\text{$i$ times}} (y_0).
\]
Since $2^{3y_i} \le y_{i+1}\le 2^{3y_i+1}-1$ for each $i$, it follows by induction that $a_i\le y_i\le b_i$ for each $i$ (if $y_i$ exists). 

Now fix some $i\ge  1$. Since the event $A_{i-1}$ is determined by $y_1,\ldots, y_{i-1}$, and these random variables can take only finitely many values (by the above paragraph), we can write $A_{i-1}$ as a finite union of events of the form $\{y_1=c_1,\ldots, y_{i-1}=c_{i-1}\}$, where $c_1< c_2 <\cdots < c_{i-1}\in \mathbb{N}$. 

Now note that for any $c_1<\cdots < c_{i-1}$, the event $A_i\cap \{y_1=c_1,\ldots, y_{i-1}=c_{i-1}\}$ happens if and only if $\{y_1=c_1,\ldots, y_{i-1}=c_{i-1}\}$ happens and there is some $y\in \{2^{3c_{i-1}}, 2^{3c_{i-1}}+1,\ldots,2^{3c_{i-1}+1}-1\}$ that has an edge to $c_{i-1}$, but to no other number in $\{0,\ldots, 2c_{i-1}-1\}$. The event $\{y_1=c_1,\ldots, y_{i-1}=c_{i-1}\}$ is in $\mf_{c_{i-1}}$, where $\mf_x$ denotes the $\sigma$-algebra generated by the edges between all numbers in $\{0,\ldots,x\}$. On the other hand, on the event $\{y_1=c_1,\ldots, y_{i-1}=c_{i-1}\}$, it is not hard to see that
\[
\pp(A_i | \mf_{c_{i-1}}) = 1 - (1- 2^{-2c_{i-1}})^{2^{3c_{i-1}}}. 
\]
Thus,
\begin{align*}
&\pp(A_i\cap \{y_1=c_1,\ldots, y_{i-1}=c_{i-1}\}) \\
&= \pp(y_1=c_1,\ldots, y_{i-1}=c_{i-1})(1 - (1- 2^{-2c_{i-1}})^{2^{3c_{i-1}}})\\
&\ge   \pp(y_1=c_1,\ldots, y_{i-1}=c_{i-1})(1 - e^{- 2^{c_{i-1}}}),
\end{align*}
where in the last step we used the inequality $0\le 1-x\le e^{-x}$ (which holds for all $x\in [0,1]$). 
Note that the term inside the parentheses on  the right side is an increasing function of $c_{i-1}$, and the maximum possible value of $y_{i-1}$ is $b_{i-1}$. Thus, summing both sides over all values of $c_1,\ldots,c_{i-1}$ such that $\{y_1=c_1,\ldots, y_{i-1}=c_{i-1}\}\subseteq A_{i-1}$, we get
\begin{align*}
\pp(A_i) = \pp(A_i \cap A_{i-1}) &\ge \pp(A_{i-1}) (1 - e^{- 2^{b_{i-1}}}). 
\end{align*}
Proceeding inductively, this gives
\begin{align*}
\pp(A_1\cap\cdots \cap A_i) \ge \prod_{k=0}^{i-1} (1-e^{-2^{b_k}}). 
\end{align*}
Taking $i\to\infty$, we get
\begin{align*}
\pp(B)\ge  \prod_{k=0}^\infty (1-e^{-2^{b_k}}), 
\end{align*}
where 
\[
B := \bigcap_{k=1}^\infty A_k. 
\]
Now recall that the event $B$, as well as the numbers $b_0,b_1,\ldots$, are dependent on our choice of $y_0$. To emphasize this dependence, let us write them as $B(y_0)$ and $b_k(y_0)$. Then by the above inequality,
\begin{align*}
\sum_{y_0\in \mathbb{N}} \pp(B(y_0)^c) \le \sum_{y_0\in \mathbb{N}} \biggl(1-\prod_{k=0}^\infty (1-e^{-2^{b_k(y_0)}})\biggr),
\end{align*}
where $B(y_0)^c$ denotes the complement of $B(y_0)$. Due to the extremely rapid growth of $b_k(y_0)$ as $k\to\infty$, and the fact that $b_0(y_0) = y_0$, it is not hard to see that the right side is finite. Therefore, by the Borel--Cantelli lemma, $B(y_0)^c$ happens for only finitely many $y_0$ with probability one. In particular, with probability one, $B(y_0)$ happens for some $y_0$. This completes the proof.
\end{proof}
We can now prove Theorem \ref{xthm}.
\begin{proof}[Proof of Theorem \ref{xthm}]
Fix a realization of $G(\infty, 1/2)$. Let $x$ be so large that 
\[
\mu([x,\infty)) < \frac{1}{10},
\]
and 
\[
\prod_{k=1}^\infty (1-2^{-a_k(x)+1})\ge \frac{9}{10}. 
\] 
Let $x_0,x_1,x_2,\ldots$ be a sequence having the properties listed in Lemma \ref{xlmm} (which exists with probability one, by the lemma). Discarding some initial values if necessary, let us assume that $x_0 > x$. By the listed properties, it is obvious that $x_i\to \infty$ as $i\to\infty$. Thus, to prove Theorem \ref{xthm}, it suffices to prove that 
\begin{align}\label{tauxi}
\liminf_{i\to\infty} \frac{\tau_{x_i}}{\log_{16}^* x_i} \ge 1.
\end{align}
We will now deduce this from the properties of the sequence. 

Suppose that our random walk starts from $x_i$ for some $i\ge 1$. Since $x_i$ connects to $x_{i-1}$ by an  edge, but not to any other number in $\{0,\ldots, 2x_{i-1}-1\}$, we see that the probability of the walk landing up at $x_{i-1}$ in the next step is at least
\[
1 - \frac{1}{2^{-x_i}}\sum_{k= 2x_i}^\infty 2^{-k} = 1- 2^{-x_i + 1}. 
\]
Proceeding by induction, this shows that the chance that the walk lands up at $x_0$ at step $i$ is at least
\[
\prod_{k=1}^i(1- 2^{-x_k + 1}).
\]
Let $\mu_i$ be the law of walk at step $i$ (starting from $x_i$, and conditional on the fixed realization of our random graph). Then by the above deduction and the facts that $x_0>x$ and $x_k \ge a_k(x_0)\ge a_k(x)$, we have
\begin{align*}
\mu_i([x,\infty)) &\ge \prod_{k=1}^i(1- 2^{-x_k + 1})\ge \prod_{k=1}^i (1-2^{-a_k(x)+1}). 
\end{align*}
By our choice of $x$, the last expression is bounded below by $9/10$. But $\mu([x,\infty))< 1/10$. Thus, the total variation distance between $\mu_i$ and $\mu$ is at least $8/10$. In particular, $\tau_{x_i} > i$. Now, 
\begin{align*}
x_i &\le 2^{3x_{i-1} + 1}-1 \le 16^{x_{i-1}},
\end{align*}
which shows that $\log_{16}^*x_i \le \log_{16}^* x_{i-1}  + 1$. 
Proceeding inductively, we get 
\[
\log_{16}^*x_i \le i + \log_{16}^* x_0.
\]
Thus, $\tau_{x_i} > \log_{16}^* x_i - \log_{16}^* x_0$. This proves  \eqref{tauxi}. 
\end{proof}

\section{The upper bound (assuming a spectral gap)}\label{upperproof}
This section gives the upper bound for both the binary and random model of the Rado graph. Indeed, the proof works for  a somewhat general class of graphs and more general base measures $Q$. The argument assumes that we have a spectral gap estimate. These are proved below in Sections \ref{spectralgap} and \ref{spectralgap2}. We give this part of the argument first because, as with earlier sections, it gives a useful picture of the random graph.

Take any undirected graph on the nonnegative integers, with the following property. 
\begin{align}\label{assump}
\begin{rcases}
\text{There exists $C>0$ such that for any $j\ge 2$,}\\
\text{$j$ is connected to some $k\le C\log j$.}  
\end{rcases}
\end{align}
Let $\{X_n\}_{n\ge 0}$ be the Markov chain on this graph, which, starting at state $i$, jumps to a neighbor $j$ with probability proportional to $Q(j) = 2^{-(j+1)}$. The following is the main result of this section.
\begin{thm}\label{upperthm}
Let $K$ be the transition kernel of the Markov chain defined above. Suppose that $K$ has a spectral gap. Let $\mu$ be the stationary distribution of the chain, and let $a:= e^{1/C}$. Then for any $i\in \N$ and any $\ell\ge 1$, 
\begin{align*}
\|K_i^\ell - \mu\|\le C_1e^{\log_a^*i} e^{-C_2 \ell},
\end{align*}
where $C_1$ and $C_2$ are positive constants that depend only the properties of the chain (and not on $i$ or $\ell$).
\end{thm}

By Lemma \ref{lem5}, $G(\infty, 1/2)$ satisfies the property \eqref{assump} with probability one, for some $C$ that may depend on the realization of the graph. The Rado graph also satisfies property \eqref{assump}, with $K = 1/\log 2$. Thus, the random walk starting from $j$ mixes in time $\log_2^* j$ on the Rado graph, provided that it has a spectral gap. For $G(\infty, 1/2)$, assuming that the walk has a spectral gap, the mixing time starting from $j$ is $\log_a^* j$, where $a$ depends on the realization of the graph. The spectral gap for $G(\infty, 1/2)$ will be proved in Section \ref{spectralgap}, and the spectral gap for the Rado graph will be established in Section \ref{spectralgap2}. Therefore, this proves Theorem \ref{mainthm} and also the analogous result for $G(\infty, 1/2)$.

\begin{proof}[Proof of Theorem \ref{upperthm}]
Note that $a>1$. Let $Z_n := \log_a^* X_n$. We claim that there is some $j_0$ sufficiently large, and some positive constant $c$, such that 
\begin{align}\label{zineq}
\ee(e^{Z_{n+1}}|\mf_n) \le e^{Z_n-c} \ \ \ \text{ if } Z_n >j_0,
\end{align}
where $\mf_n$ is the $\sigma$-algebra generated by $X_0,\ldots,X_n$. (The proof is given below.) This implies that if we define the stopping time $S := \min\{n\ge 0: X_n\le j_0\}$, then $\{e^{Z_{S\wedge n} + c(S\wedge n)}\}_{n\ge 0}$ is a supermartingale with respect to the filtration $\{\mf_{n}\}_{n\ge 0}$ (see details below). Moreover, it is nonnegative. Thus, if we start from the deterministic initial condition $X_0 = j$, then for any $n$, 
\begin{align*}
\ee(e^{Z_{S\wedge n} + c (S\wedge n)}|X_0=j) \le e^{Z_{S\wedge 0} + c (S\wedge 0)} = e^{\log_a^* j}.
\end{align*}
But $Z_{S\wedge n}\ge 0$. Thus, $\ee(e^{c(S\wedge n)}|X_0=j)\le e^{\log_a^* j}$. Taking $n\to \infty$ and applying the monotone convergence theorem, we get 
\begin{align}\label{sbd}
\ee(e^{cS}|X_0=j)\le e^{\log_a^* j}.
\end{align}
Now take any $j\ge 1$ and $n\ge 1$.  Let $\mu$ be the stationary distribution, and let $\mu_{j,n}$ be the law of $X_n$ when $X_0=j$. Take any $A\subseteq \{0,1,\ldots\}$. Then for any $m\le n$,
\begin{align*}
&\mu_{j,n}(A) = \pp(X_n \in A|X_0=j) \\
&= \sum_{i=0}^m \sum_{l=0}^{j_0} \pp(X_n\in A|S=i, \, X_i = l,\, X_0=j) \pp(S=i,\, X_i=l|X_0=j) \\
&\qquad +  \pp(X_n\in A| S > m, \, X_0=j)\pp(S>m|X_0=j).
\end{align*}
But 
\begin{align*}
\pp(X_n\in A|S=i, \, X_i = l, \, X_0=j) &= \pp(X_n \in A|X_i = l) \\
&= \mu_{l,n-i}(A),
\end{align*}
and 
\begin{align*}
\mu(A) &= \sum_{i=0}^m\sum_{l=0}^{j_0}\mu(A)\pp(S=i, \, X_i=l|X_0=j)  \\
&\qquad + \mu(A)\pp(S > m|X_0=j).
\end{align*}
Thus, 
\begin{align*}
&|\mu_{j,n}(A)-\mu(A)| \\
&\le \sum_{i=0}^m\sum_{l=0}^{j_0} |\mu_{l,n-i}(A) - \mu(A)|\pp(S=i, \, X_i=l|X_0=j) +\pp(S>m |X_0=j).
\end{align*}
Now, if our Markov chain has a spectral gap, there exist constants $C_1$ and $C_2$ depending only on $j_0$ and the spectral gap, such that
\[
 |\mu_{l,n-i}(A) - \mu(A)| \le C_1 e^{-C_2(n-i)}\le C_1 e^{-C_2(n-m)}
\]
for all $0\le i\le m$ and $0\le l\le j_0$. Using this bound 
and the bound \eqref{sbd} on $\ee(e^{cS}|X_0=j)$ obtained above, we get
\begin{align*}
|\mu_{j,n}(A)-\mu(A)| &\le C_1 e^{-C_2(n-m)} + e^{\log_a^* j - cm}.
\end{align*}
Taking $m = \lceil n/2\rceil$, we get the desired result.
\end{proof}
\begin{proof}[Proof of inequality \eqref{zineq}]
It suffices to take $n=0$. Suppose that $X_0=j$ for some $j\ge 1$. By assumption, there is a neighbor $k$ of $j$ such that $k\le K\log j = \log_a j$. Assuming that $j$ is sufficiently large (depending on $K$), we have that for any $l\le k$, 
\[
\log_a^* l \le \log_a^* k \le \log_a^*(\log_a j) = \log_a^*j -1.
\]
Also, $\log_a^*l \le \log_a^* j$ for any $l\le j$. Thus,
\begin{align*}
\ee(e^{Z_1 - Z_0}|X_0=j) &\le e^{-1} \cdot \pp(X_1 \le k |X_0=j) \\
&\qquad + \pp(k< X_1\le j | X_0=j)\\
&\qquad + \sum_{l> j} e^{\log_a^* l - \log_a^*j}\pp(X_1=l| X_0=j).
\end{align*}
Now for any $l \ge k$,
\begin{align*}
\pp(X_1=l | X_0=j) &\le \frac{\pp(X_1=l | X_0=j)}{\pp(X_1 = k | X_0=j)} \\
&= \frac{Q(l)}{Q(k)} = 2^{-(l-k)}. 
\end{align*}
Thus,
\begin{align*}
\sum_{l>j} e^{\log_a^*l - \log_a^*j} \pp(X_1=l | X_0=j)&\le \sum_{l > j} e^{\log_a^*l - \log_a^*j}  2^{-(l-k)},
\end{align*}
which is less than $1/4$ if $j$ is sufficiently large (since $k\le \log_a^*j$). Next, let $L$ be the set of all $l>k$ that are connected to $j$. Then 
\begin{align*}
\pp(X_1 > k | X_0=j) &\le \frac{\pp(X_1 > k | X_0=j) }{\pp(X_1 \ge  k | X_0=j) } = \frac{\sum_{l \in L} 2^{-l}}{2^{-k}+ \sum_{l\in L} 2^{-l}}. 
\end{align*}
Since the map $x\mapsto x/(2^{-k} +x)$ is increasing, this shows that
\begin{align*}
\pp(X_1 > k | X_0=j) &\le \frac{\sum_{l > k} 2^{-l}}{2^{-k}+ \sum_{l> k} 2^{-l}} = \frac{1}{2}.
\end{align*}
Combining, we get that for sufficiently large $j$, 
\begin{align*}
\ee(e^{Z_1 - Z_0}|X_0=j) &\le e^{-1} \pp(X_1\le k|X_0=j) + \pp(X_1 > k|X_0=j) + \frac{1}{4}\\
&= e^{-1} + (1-e^{-1}) \pp(X_1 > k|X_0=j) + \frac{1}{4}\\
&\le e^{-1} + \frac{1-e^{-1}}{2} + \frac{1}{4} = \frac{3+2e^{-1}}{4} < 1.
\end{align*}
\end{proof}
\begin{proof}[Proof of the supermartingale property]
Note that
\begin{align*}
&\ee(e^{Z_{S\wedge (n+1)} + c(S\wedge (n+1))}|\mf_{n})\\
&= \sum_{i=0}^n \ee(e^{Z_{S\wedge (n+1)} + c(S\wedge (n+1))} 1_{\{S=i\}}|\mf_{n})\\
&\qquad + \ee(e^{Z_{S\wedge (n+1)} + c(S\wedge (n+1))} 1_{\{S> n\}}|\mf_{ n})\\
&= \sum_{i=0}^n \ee(e^{Z_{i} + c i} 1_{\{S=i\}}|\mf_{n})+ \ee(e^{Z_{n+1} + c(n+1)} 1_{\{S> n\}}|\mf_{ n}).
\end{align*}
Now, the events $\{S=i\}$ are $\mf_{n}$-measurable for all $0\le i\le n$, and so is the event $\{S>n\}$. Moreover, $Z_0,\ldots, Z_n$ are also $\mf_{n}$-measurable. Thus, the above expression shows that
\begin{align*}
&\ee(e^{Z_{S\wedge (n+1)} + c(S\wedge (n+1))}|\mf_{n})\\
&= 1_{\{S\le n\}} e^{Z_{S\wedge n} + c(S\wedge n)} + 1_{\{S> n\}}\ee(e^{Z_{n+1} + c(n+1)} |\mf_{ n}).
\end{align*}
But if $S>n$, then $Z_n>j_0$, and therefore by \eqref{zineq},
\[
\ee(e^{Z_{n+1} + c(n+1)} |\mf_{ n})\le e^{Z_n - c + c(n+1)} = e^{Z_n + cn}.
\]
Thus,
\begin{align*}
&\ee(e^{Z_{S\wedge (n+1)} + c(S\wedge (n+1))}|\mf_{n})\\
&\le 1_{\{S\le n\}} e^{Z_{S\wedge n} + c(S\wedge n)} + 1_{\{S> n\}}e^{Z_n + cn}\\
&= e^{Z_{S\wedge n} + c(S\wedge n)}.
\end{align*}
\end{proof}

\section{Spectral gap for the random model}\label{spectralgap}
Our next goal  is to show that the random reversible couple  $(K,\pi)$ admits a spectral gap.  The arguments make use of the ideas and notation of Section \ref{preliminaries}. In particular, recall the event $\cB = \{p(x) < x \ \ \fo x\in \N_+\}$ from Lemma \ref{lem1} and the random tree $\cT$ with edge set $F$ from Lemma \ref{lem4}. The argument uses a version of Cheeger's inequality for trees which is further developed in Appendix~\ref{a1}.
\begin{prop}\label{pro6}
On $\cB$, 
there exists a random constant $\Lambda>0$ such that
\bq
\fo f\in L^2(\pi),\qquad \Lambda\pi[(f-\pi[f])^2]&\leq &\cE(f)\eq
where in the r.h.s.\ $\cE$ is the Dirichlet form defined by
\bq
\fo f\in L^2(\pi),\qquad \cE(f)&\df&
 \f12\sum_{x,y\,\in\,\N}(f(y)-f(x))^2\, \pi(x)K(x,y).\eq
\end{prop}
\par
Taking into account that  for any $f\in L^2(\pi)$, the variance $\pi[(f-\pi[f])^2]$ of $f$ with respect to $\pi$ is bounded above by $\pi[(f-f(0))^2]$, the previous result is 
an immediate consequence of the following existence of positive first Dirichlet eigenvalue under $\cB$.
\begin{prop}\label{pro7}
On $\cB$, 
there exists a random constant $\Lambda>0$ such that
\bqn{D}
\fo f\in L^2(\pi),\qquad \Lambda\pi[(f-f(0))^2]&\leq &\cE(f)\eqn
\end{prop}
\par
The proof of Proposition \ref{pro7} is based on the pruning of $\cG$ into $\cT$ and then resorting to Cheeger's inequalities for trees. More precisely, let us introduce the following notations. Define the Markov kernel $K_{\cT}$ as 
\bq
\fo x,y\in\N,\qquad
K_{\cT}(x,y)&\df& 
\lt\{\begin{array}{ll}
K(x,y)&\hbox{if $\{x,y\}\in F$,}\\
1-\sum_{z\in \N \setminus\{ x\}}K_{\cT}(x,z)&\hbox{if $x=y$,}\\
0&\hbox{otherwise.}\end{array}\rt.\eq
Note that this kernel is reversible with respect to $\pi$.
The corresponding Dirichlet form is given by
\bq
\fo f\in L^2(\pi),\qquad \cE_{\cT}(f)&\df&
 \f12\sum_{x,y\,\in\,\N}(f(y)-f(x))^2\, \pi(x)K_{\cT}(x,y)\\
 &=& \sum_{\{x,y\}\in F}(f(y)-f(x))^2\, \pi(x)K(x,y)
 \eq
 It will be convenient to work with
 \bq
 \wi\cE&\df& Z\cE_{\cT}\eq
where $Z$ is the normalizing constant of $\pi$, as in equation \eqref{piform}. Define a nonnegative measure $\mu$ on $\N_+$ as 
\bqn{mu}
\fo x\in \N_+,\qquad \mu(x)&\df& 
Q(x)Q(p(x)).\eqn
Then we have:
\begin{prop}\label{pro8}
On $\cB$, there exists $\lambda>0$ such that
\bqn{d}
\fo f\in L^2(\mu),\qquad \lambda\mu[(f-f(0))^2]&\leq &\wi\cE(f)\eqn
\end{prop}
\par
This result immediately implies Proposition \ref{pro7}. Indeed, 
due on one hand to the inclusion $N(x)\subset \lin p(x),\iy\lin$ and on the 
other hand to
 the nature of $Q$, we have
\bqn{QQQ}
\fo x\in\N_+,\qquad Q(p(x))\ \leq \ Q(N(x))\ \leq \ 2Q(p(x))\eqn
Thus for any $f\in L^2(\mu)$,
\bq
\lambda\pi[(f-f(0))^2]&= &\f{\lambda }Z\sum_{x\in\N_+} (f(x)-f(0))^2 Q(x)Q(N(x))\\
&\leq &\f{2\lambda }{Z}\sum_{x\in\N_+} (f(x)-f(0))^2 Q(x)Q(p(x))\\
&=&\f{2\lambda }{Z}\mu[(f-f(0))^2]\leq \f{2 }{Z}\wi\cE(f)= 2\cE_{\cT}(f)\le2 \cE(f),\eq
and thus, Proposition \ref{pro7} holds with $\Lambda\df\lambda/2$.\par
\me
The proof of Proposition \ref{pro8} is based on a Dirichlet-variant of the Cheeger inequality (which is in fact slightly simpler than the classical one, see Appendix~\ref{a1}). For any $A\subset \N_+$, define $\pa A \df\{\{x,y\}\st x\in A, y\not\in A\}\subset \bar E$. Endow $\bar E$ with the measure $\nu$ induced by 
\bq
\fo \{x,y\}\in \bar E,\qquad \nu(\{x,y\})&\df& Z\pi(x)K_{\cT}(x,y)\\
&=&\lt\{\begin{array}{ll}Q(x)Q(y)&\hbox{if $\{x,y\}\in F$,}\\
0&\hbox{otherwise.}\end{array}\rt.
\eq
Define
the Dirichlet--Cheeger constant
\bq
\iota&\df& \inf_{A\in \cA}\f{\nu(\pa A)}{\mu(A)}\ \geq \ 0\eq
where 
\bq
\cA&\df& \{A\subset \N_+\st A\neq \emptyset\}.\eq
The proof of the traditional Markovian Cheeger's inequality given in the lectures by Saloff-Coste \cite{MR1490046} implies directly that the best constant $\lambda$ in Proposition \ref{pro8} satisfies
\bq
\lambda&\geq & \f{\iota^2}{2}\eq
Thus it remains to check:
\begin{prop}\label{pro9}
On $\cB$, we have $\iota\geq 1/2$ and in particular $\iota>0$.
\end{prop}
\begin{proof}
Take any nonempty $A\in\cA$ and decompose it into its connected components with respect to $\cT$:
\bq
A&=&\bigsqcup_{i\in\cI} A_i\eq
where the index set $\cI$ is at most denumerable. Note that
\bq
\mu(A)=\sum_{i\in\cI} \mu(A_i), \ \ 
\nu(A)=\sum_{i\in\cI}\nu(A_i),\eq
where the second identity holds because there are no edges in $F$ connecting two different $A_i$'s. 
Thus, it follows that
\bq
\iota&=& \inf_{A\in \wi\cA}\f{\nu(\pa A)}{\mu(A)},\eq 
where $\wi\cA$ is the set of subsets of $\cA$ which are $\cT$-connected.
\par
Consider $A\in\wi\cA$, it has a smallest element $a\in\N_+$ (since $0\not\in A$).
Let $T_a$ be the subtree of descendants of $A$ in $\cT$ (i.e., the set of vertices from $\N_+$ whose non-self-intersecting path to $0$ passes through $a$).
We have 
\bq
A&\subset& T_a,\\
\pa A&\supset& \{a,p(a)\}\ =\ \pa T_a,\eq
and it follows that
\bq
\f{\nu(\pa A)}{\mu(A)}&\geq & \f{\nu(\pa T_a)}{\mu(T_a)}.\eq
We deduce  that
\bq
\iota&\ge &\inf_{a\in\N_+}\f{\nu(\pa T_a)}{\mu(T_a)} =
\inf_{a\in\N_+}\f{Q(a)Q(p(a))}{\mu(T_a)}.\eq
On $\cB$, we have for any $a\in \N_+$, on the one hand
\bqn{dec1}
\fo x\in T_a,\qquad p(x)&\geq &p(a),\eqn
and on the other hand
\bqn{dec2}
T_a&\subset& \lin a,\iy\lin.\eqn
We get
\bq
\mu(T_a)&= & \sum_{x\in T_a}Q(x)Q(p(x))\\
&\geq & Q(p(a))\sum_{x\in T_a}Q(x)\\
&\geq &Q(p(a))\sum_{x\in \lin a,\iy\lin}Q(x)\\
&=&2Q(p(a))Q(a).\eq
It follows that \bq
\iota&\geq &\f12.\eq
\end{proof}
 Lemma  \ref{lem5} can now be used to see that the ball Markov chain on the random graph has a.s.\ a spectral gap.
Indeed, we deduce from Lemma  \ref{lem5} that there exists a (random) vertex $x_0\in \N$ such that
\bq
\fo x> x_0,\qquad p(x)&<&x.\eq
Consider 
\bq
x_1&\df& \max\{p(x)\st x\in\lin 1,x_0\rin\}.\eq
It follows that for any $a>x_1$, we have
\bq
\fo x\in T_a,\qquad p(x)&<&x.\eq
(To see this, take any path $a_0,a_1,\ldots$ in $T_a$, starting at $a_0=a$, so that $p(a_i)=a_{i-1}$ for each $i$. Let $k$ be the first index such that $a_{k} \ge a_{k+1}$, assuming that there exists such a $k$. Then $a_{k+1}\le x_0$, and so $a_k = p(a_{k+1})\le x_1$. But this is impossible, since $a_0\le a_k$ and $a_0>x_1$.) 

In particular, we see that \eqref{dec1} and \eqref{dec2} hold for $a>x_1$. As a consequence, we get
\bq
\inf_{a> x_1}\f{\nu(\pa T_a)}{\mu(T_a)}&\geq & \f12\eq
By the finiteness of $\lin 1, x_1\rin$, we also have
\bq
\inf_{a\in \lin 1, x_1\rin}\f{\nu(\pa T_a)}{\mu(T_a)}&> & 0.\eq
So, finally, we have
\bq
\iota&=&\inf_{a\in\N_+}\f{\nu(\pa T_a)}{\mu(T_a)}\ >\ 0,\eq
which shows that $G(\infty, 1/2)$ has a spectral gap a.s.

\section{Spectral gap for the Rado graph}\label{spectralgap2}
This section proves the needed spectral gap for the Rado graph. Here the graph has vertex set $\N$ and an edge from $i$ to $j$ if $i$ is less than $j$ and the $i$th bit of $j$ is a one. We treat carefully the case of a more general base measure, $Q(x) = (1-\delta)\delta^x$. As delta tends to $1$, sampling from this $Q$ is a better surrogate for ``pick a neighboring vertex uniformly''. Since the normalization doesn't enter, throughout take $Q(x)=\delta^x$.  The heart of the argument is a discrete version of Hardy's inequality for trees. This is developed below with full details in Appendix \ref{a2}.

Consider the transition kernel $K$ reversible with respect to $\pi$ and associated to the measure $Q$ given by
\bq
\fo x\in\N,\qquad Q(x)&\df& \delta^x\eq
where $\delta\in(0,1)$ (instead of $\delta=1/2$ as in the introduction, up to the normalization).  Recall that 
\bq
\fo x,y\in \N,\qquad K(x,y)&\df& \f{Q(y)}{Q(N(x))}\un_{N(x)}(y)\\
\fo x\in\N,\qquad \pi(x)&=&Z^{-1}Q(x)Q(N(x))\eq
where $N(x)$ is the set of neighbors of $x$ induced by $K$ and where $Z>0$ is the normalizing constant. 

Here is the equivalent of Proposition \ref{pro8}:
\begin{prop}\label{dHardy}
We have 
\bq
 \lambda&\geq & \f{1-\delta}{16(2\vee\lceil \log_2\log_2(2/\log_2(1/\delta))\rceil)}\eq
\end{prop}
\par
This bound will be proved via Hardy's inequalities. 
If we resort to Dirichlet--Cheeger, we rather get
\bqn{Ch}
\lambda&\geq & \f{(1-\delta)^2}2\eqn
To see the advantage of Proposition \ref{dHardy},
let $\delta$ come closer and closer to $1$, namely, approach the problematic case of ``pick a neighbor uniformly at random''.
In this situation, the  r.h.s.\ of the bound of Proposition \ref{dHardy} is of order
\bq
 \f{1-\delta}{16\lceil \log_2\log_2(1/(1-\delta))\rceil}\eq
  which is better than \eqref{Ch} as $\delta$ goes to $1-$.


Here we present the Hardy's inequalities method to get Proposition \ref{dHardy} announced above. Our goal is to show that $K$ admits a positive first Dirichlet eigenvalue:
\begin{prop}\label{pro11}
There exists $\Lambda>0$ depending on $\delta\in (0,1)$ such that
\bq
\fo f\in L^2(\pi),\qquad \Lambda\pi[(f-f(0))^2]&\leq & \f12\sum_{x,y\,\in\,\N}(f(y)-f(x))^2\, \pi(x)K(x,y)\eq
\end{prop}
\par
It follows that the reversible couple $(K,\pi)$ admits a spectral gap bounded below by $\Lambda$ given above. 
Indeed, it is an immediate consequence of the fact that for any $f\in L^2(\pi)$, the variance of $f$ with respect to $\pi$ is bounded above by $\pi[(f-f(0))^2]$.

The proof of Proposition \ref{pro11} is based on a pruning of $K$ and Hardy's inequalities for trees. Consider the set of unoriented edges induced by $K$:
\bq
E&\df&\{\{x,y\}\in \N \times\N \st  K(x,y)>0\}.\eq
(In particular, $E$ does not contain the self-edges or singletons.) For any $x\in\N_+$, let $p(x)$ the smallest bit equal to $1$ in the binary expansion of $x$, i.e.,
\bq
p(x)&\df&\min\{y\in \N\st K(x,y)>0\}.\eq
Define the subset $F$ of $E$ by
 \bq
F&\df&\{\{x,p(x)\}\in E\st x\in\N_+\}\eq
and the function $\nu$ on $F$ via
\bq
\fo \{x,p(x)\}\in F,\qquad \nu(\{x,p(x)\})&\df& Z\pi(x)K(x,p(x))\\
&=&Q(x)Q(p(x)).\eq
To any $f\in L^2(\pi)$, associate the function $(df)^2$ on $F$ given by
\bq
\fo \{x,p(x)\}\in F,\qquad (df)^2(\{x,p(x)\})&\df&(f(x)-f(p(x)))^2.\eq
Finally, consider the (non-negative) measure $\mu$ defined on $\N_+$ via
\bqn{muH}
\fo x\in \N_+,\qquad \mu(x)&\df& 
Q(x)Q(p(x)).\eqn
Then we have:
\begin{prop}\label{pro12}
There exists $\lambda>0$ depending on $\delta\in (0,1)$ such that
\bq
\fo f\in L^2(\mu),\qquad \lambda\mu[(f-f(0))^2]&\leq & \sum_{e\in F}(df)^2(e) \nu(e).\eq
\end{prop}
This result immediately implies Proposition \ref{pro11}. Indeed, note that by the definition of $Q$,
\bqn{QQQH}
\fo x\in\N_+,\qquad Q(p(x))\ \leq \ Q(N(x))\ \leq \ \f1{1-\delta}Q(p(x)).\eqn
Thus, for any $f\in L^2(\mu)$,
\bq
\lambda\pi[(f-f(0))^2]&= &\f{\lambda }Z\sum_{x\in\N_+} (f(x)-f(0))^2 Q(x)Q(N(x))\\
&\leq &\f{\lambda }{(1-\delta)Z}\sum_{x\in\N_+} (f(x)-f(0))^2 Q(x)Q(p(x))\\
&=&\f{\lambda }{(1-\delta)Z}\mu[(f-f(0))^2]\\
&\leq & \f1{(1-\delta)Z}\sum_{e\in F}(df)^2(e) \nu(e)\\
&\leq & \f1{2(1-\delta)}\sum_{x,y\,\in\,\N}(f(y)-f(x))^2\, \pi(x)K(x,y)\eq
namely Proposition \ref{pro11} holds with $\Lambda\df \lambda(1-\delta)$.\par
\me
Note that $\N$ endowed with the set of non-oriented edges $F$ has the structure of a tree. We interpret $0$ as its root, so that for any $x\in\N_+$, $p(x)$ is the parent of $x$.
Note that for any $x\in\N$, the children of $x$ are exactly the numbers $y2^x$, where $y$ is an odd number. We will denote $h(x)$ the height of $x$ with respect to the root $0$
(thus, the odd numbers are exactly the elements of $\N$ whose height is equal to $1$).

According to \cite{MR1709277} (see also  Evans, Harris and Pick \cite{MR1345719}), the best constant $\lambda$ in Proposition~\ref{pro12}, say $\lambda_0$, can be estimated up to a factor 16 via Hardy's inequalities for trees, see \eqref{HH} below. To describe them we need several notations.

Let $\cT$ the set of subsets $T$ of $\N_+$ satisfying the following conditions
\begin{itemize}
\item $T$ is non-empty and connected (with respect to $F$),
\item $T$ does not contain $0$,
\item there exists $M\geq 1$ such that $h(x)\leq M$ for all $x\in T$,
\item if $x\in T$ has a child in $T$, then all children of $x$ belong to $T$.
\end{itemize}
Note that any $T\in\cT$ admits a closest element to $0$, call it $m(T)$. Note that $m(T)\neq 0$.
When $T$ is not reduced to the singleton $\{m(T)\}$, then $T\setminus\{m(T)\}$ has a denumerable infinity of connected components which are indexed by the children of $m(T)$.
Since these children are exactly the $y2^{m(T)}$, where $y\in \cI$,  the set of odd numbers, call $T_{y2^{m(T)}}$ the connected component of $T\setminus\{m(T)\}$ associated to $y2^{m(T)}$.
Note that $T_{y2^{m(T)}}\in \cT$.

We extend $\nu$ as a functional on $\cT$, via the iteration
\begin{itemize}
\item when $T$ is the singleton $\{m(T)\}$, we take $\nu(T)\df \nu(\{m(T),p(m(T))\})$,
\item when $T$ is not a singleton, decompose $T$ as $\{m(T)\}\sqcup \bigsqcup_{y\in\cI}T_{y2^{m(T)}}$, then $\nu$ is defined as
\bqn{recH}
\f1{\nu(T)}&=&\f1{\nu(\{m(T)\})}+\f1{\sum_{y\in\cI} \nu(T_{y2^{m(T)}})}\eqn
\end{itemize}
For $x\in\N_+$, let $S_x$ be the set of vertices $y\in\N_+$ whose path to $0$ passes through $x$.
For any $T\in \cT$ we associate the subset 
\bq
T^*&\df& (S_{m(T)}\setminus T)\sqcup L(T)\eq
where $L(T)$ is the set of leaves of $T$, namely the $x\in T$ having no children in $T$. Equivalently, $T^*$ is the set of all descendants of the leaves of $T$, themselves included.

Consider $\cS\subset \cT$ the set of $T\in\cT$ which are such that $m(T)$ is an odd number. Finally, define
\bq
A&\df& \sup_{T\in\cS} \f{\mu(T^*)}{\nu(T)}.\eq
We are interested in this quantity because of the Hardy inequalities:  
\bqn{HH}
A\ \leq \ \f1{\lambda_0}\ \leq\ 16A,\eqn
where recall that $\lambda_0$ is the best constant in Proposition~\ref{pro12}. 
(In \cite{MR1709277}, only finite trees were considered, the extension to infinite trees is given in Appendix \ref{a2}) So, to prove Proposition \ref{pro12}, it is sufficient to show that $A$ is finite. To investigate $A$, we need some further definitions. For any $x\in \N_+$, let
\bq
b(x)&\df&\f{Q(2^x)}{Q(p(x))}.\eq
A finite path from $0$ in the direction to infinity is a finite sequence $z\df(z_n)_{n\in\lin 0, N\rin}$ of elements of $\N_+$ such that
$z_0=0$ and $p(z_n)=z_{n-1}$ for any $n\in \lin 1, N\rin$. On such a path $z$, we define the quantity
\bq
B(z)&\df&\sum_{n\in\lin 1,N\rin} b(z_n).\eq
The following technical result will be crucial for our purpose of showing that $A$ is finite.
\begin{lmm}\label{lem13}
For any  finite path from $0$ in the direction to infinity  $z\df(z_n)_{n\in\lin 0, N\rin}$, we have
\bq
B(z)&\leq & C,\eq
where
\bq C&\df& \sum_{l\in\N} {\delta^{2^{2^l}-l}}\ <\ +\iy.\eq
\end{lmm}
\begin{proof}
Note that for any
$n\in\lin 1,N\rin$, 
 $h(z_n)=n$.
Furthermore, for any
$x\in\N_+$, we have $h(x)\leq x$ and we get
$h(p(z_n))=h(z_n)-1= n-1$, so that $p(z_n)\geq n-1$.
Writing $z_n=y_n2^{p(z_n)}$, for some odd number $y_n$,
it follows that
\bq
b(z_n)&=&\f{Q(2^{y_n2^{p(z_n)}})}{Q(p(z_n))}\\
&=&{\delta^{2^{y_n2^{p(z_n)}}-p(z_n)}}\\
&\leq &{\delta^{2^{2^{p(z_n)}}-p(z_n)}}\\
&\leq &{\delta^{2^{2^{n-1}}-n-1}}.\eq
The desired result follows at once.
\end{proof}

We need two ingredients about ratios  $\mu(T^*)/\nu(T)$.
Here is the first one.
\begin{lmm}\label{prel1}
For any $T\in\cT$ which is a singleton, we have 
\[
\frac{\mu(T^*)}{\nu(T)}\leq \frac{1}{1-\delta}.
\]
\end{lmm} 
\begin{proof}
When $T$ is the singleton $\{m(T)\}$,
 on the one hand we have 
 \[
 \nu(T)=\nu(\{p(m(T)),m(T)\})=\mu(m(T)).
 \]
On the other hand, $T^*$ is the subtree growing from $m(T)$, namely the subtree containing all the descendants of $m(T)$.
Note two properties of $T^*$:
\bqn{Tst1}
T^*&\subset&\{y\in\N_+\st y\geq m(T)\},\\
\label{Tst2}\fo y\in T^*,\qquad p(y)&\geq &p(m(T)),\eqn
and we further have $p(y)\geq m(T)$ for any $y\in T^*\setminus\{m(T)\}$.
It follows that 
\bqn{ub}
\nonumber\mu(T^*)
\nonumber&=&\sum_{y\in T^*} Q(y)Q(p(y))\\
\nonumber&\leq & Q(p(m(T)))\sum_{y\geq m(T)} Q(y)\\
\nonumber&=&Q(p(m(T)))\sum_{y\geq m(T)} \delta^{y}\\
&= & Q(p(m(T))) \f{Q(m(T))}{1-\delta}\\
\nonumber&=&\f1{1-\delta}\mu(m(T)).\eqn
Thus, we get
\bq
\f{\mu(T^*)}{\nu(T)}&\leq &\f1{1-\delta}.\eq
\end{proof}
For the second ingredient, we need some further definitions.
The length $\ell(T)$ of $T\in\cT$ is given by
\bq
\ell(T)&\df& \max_{x\in T}h(x)-\min_{x\in T}h(x),\eq
and for any $l\in\N$, we define
\bq
\cT_l&\df&\{T\in\cT\st \ell(T)\leq l\}\eq
\par
\begin{lmm}\label{prel3}
For any $l\in\N$, we have
\bq
\sup_{T\in\cT_l}\f{\mu(T^*)}{\nu(T)}&<&+\iy.\eq
\end{lmm}
\begin{proof}
We will prove the finiteness by induction over $l\in\N$. First, note that $\cT_0$ is the set of singletons, and so Lemma \ref{prel1} implies that
\bq
\sup_{T\in\cT_0}\f{\mu(T^*)}{\nu(T)}&\leq &\f1{1-\delta}.\eq
Next, assume that the supremum is finite for some $l\in\N$ and let us show that it is also finite for $l+1$.

Consider $T\in\cT_{l+1}$, with $\ell(T)=l+1$; in particular, $T$ is not a singleton.
Decompose $T$ as $\{m(T)\}\sqcup \bigsqcup_{y\in\cI}T_{y2^{m(T)}}$ and recall the relation \eqref{recH}. Since 
\bq
T^*&=&\bigsqcup_{y\in\cI}T_{y2^{m(T)}}^*,
\eq
it follows that
\bqn{up2}
\nonumber\f{\mu(T^*)}{\nu(T)}&=&\sum_{y\in \cI}\mu(T_{y2^{m(T)}}^*)\lt(\f1{\nu(\{m(T)\})}+\f1{\sum_{y\in\cI} \nu(T_{y2^{m(T)}})}\rt)\\
\nonumber&=&\f{\sum_{y\in \cI}\mu(T_{y2^{m(T)}}^*)}{\nu(\{m(T)\})}+\f{\sum_{y\in \cI}\mu(T_{y2^{m(T)}}^*)}{\sum_{y\in\cI} \nu(T_{y2^{m(T)}})}\\
&\leq &\f{\mu(\sqcup_{y\in \cI}T_{y2^{m(T)}}^*)}{\mu(m(T))}+\sup\lt\{\f{\mu(T_{y2^{m(T)}}^*)}{\nu(T_{y2^{m(T)}})}\st y\in\cI\rt\}.
\eqn
Consider the first term on the right.  Given $y\in\cI$, the smallest possible element of $T_{y2^{m(T)}}^*$ is $y2^{m(T)}$, and we have for any $x\in T_{y2^{m(T)}}^*$,
\[
p(x)\geq p(y2^{m(T)})=m(T).
\]
Thus we have the equivalents of \eqref{Tst1} and \eqref{Tst2}:
\bqn{Tst3}
\nonumber\bigsqcup_{y\in \cI}T_{y2^{m(T)}}^*&\subset&\{y\in\N_+\st y\geq 2^{m(T)}\},\\
\fo x\in \bigsqcup_{y\in \cI}T_{y2^{m(T)}}^*,\qquad p(x)&\geq &m(T).\eqn
Following the computation \eqref{ub}, we get 
\bq
\mu\lt(\bigsqcup_{y\in \cI}T_{y2^{m(T)}}^*\rt)
&< &\f1{1-\delta}Q(m(T))Q(2^{m(T)}),
\eq
where  the  inequality is strict, because in \eqref{Tst3} we cannot have equality for all $x\in \bigsqcup_{y\in \cI}T_{y2^{m(T)}}^*$. It follows that
\bqn{prel2}
\nonumber\f{\sum_{y\in \cI}\mu(T_{y2^{m(T)}}^*)}{\mu(m(T))}&< & \f1{1-\delta}\f{Q(m(T)) Q(2^{m(T)})}{Q(m(T))Q(p(m(T)))}\\
&=& \f1{1-\delta}b(m(T))\\
\nonumber&\leq &  \f1{1-\delta}C\eqn
where $C$ is the constant introduced in Lemma \ref{lem13}. Since for any $y\in\cI$, we have $T_{y2^{m(T)}}\in\cT_l$, we deduce the desired result from the induction hypothesis.
\end{proof}
We are now ready to prove Proposition \ref{pro12}.
\begin{proof}[Proof of Proposition \ref{pro12}]
Fix some $T\in\cS$,
we are going to show that \bq
\f{\mu(T^*)}{\nu(T)}&\leq &\f{1+C}{1-\delta}\eq
where $C$ is the constant introduced in Lemma \ref{lem13}. Due to Lemma \ref{prel1}, this bound is clear if $T$ is a singleton. When $T$ is not the singleton $\{m(T)\}$, decompose $T$ as $\{m(T)\}\sqcup \bigsqcup_{y\in\cI}T_{y2^{m(T)}}$ and let us come back to \eqref{up2}.
Denote $z_1\df m(T)$ and
\bq
\epsilon &\df &\f{b(z_1)}{1-\delta}-\f{\sum_{y\in \cI}\mu(T_{y2^{m(T)}}^*)}{\mu(m(T))}\eq
which is positive according to \eqref{prel2}. Coming back to \eqref{up2}, we have shown 
\bq
\f{\mu(T^*)}{\nu(T)}&\leq & \f{b(z_1)}{1-\delta}+\f{\mu(T^*_{z_2})}{\nu(T_{z_2})}\eq
where $z_2\in\{y2^{m(T)}\st y\in\cI\}$ is such that
\bq
\sup\lt\{\f{\mu(T_{y2^{m(T)}}^*)}{\nu(T_{y2^{m(T)}})}\st y\in\cI\rt\}\leq \f{\mu(T^*_{z_2})}{\nu(T_{z_2})}+\epsilon.\eq
To get the existence of $z_2$, we used that the supremum is finite, as ensured by Lemma \ref{prel3}.

By iterating this procedure, define a 
finite path from $0$ in the direction to infinity  $z\df(z_n)_{n\in\lin 0, N\rin}$, such that
for any $n\in\lin 1, N-1\rin$,
\bq
\f{\mu(T^*_{z_n})}{\nu(T_{z_n})}&\leq & \f{b(z_n)}{1-\delta}+\f{\mu(T^*_{z_{n+1}})}{\nu(T_{z_{n+1}})}
\eq
and $T_{z_N}$ is a singleton. We have $N\leq \max\{h(x)\st x\in T\}$. We deduce that
\bq
\f{\mu(T^*)}{\nu(T)}&\leq &\f{B(z)}{1-\delta}+\f{\mu(T^*_{z_N})}{\nu(T_{z_N})}\\
&\leq &\f{C+1}{1-\delta},\eq
as desired.
\end{proof}
\par
To get a more explicit bound in terms of $\delta$, it remains to investigate the quantity $C$.
\begin{lmm}\label{lem16}
We have
\bq
C&\leq &\lt\{\begin{array}{ll}
2&\hbox{if $\delta\in (0,1/\sqrt{2}]$,}\\
1+\lt\lceil\log_2\log_2\lt(\f2{\log_2(1/\delta)}\rt)\rt\rceil&\hbox{if $\delta\in (1/\sqrt{2},1)$.}
\end{array}\rt.\eq
\end{lmm}
\begin{proof}
Consider
\bq
l_0&\df& \min(l\in\N_+\st \delta^{2^{2^l}-l}\leq 1/2).\eq
Elementary computations enable to see that
\bq
\fo l\geq 1,\qquad 
2^{2^{l+1}}-l-1\geq 2(2^{2^l}-l),\eq
so we get
\bq
\sum_{l\geq l_0}\delta^{2^{2^l}-l}&\leq &\sum_{n\geq 0}\f1{2^{2^n}}\\
&\leq & \sum_{n\geq 1}\f1{2^{n}}\\
&=&1.\eq
Since we have
\bq
\fo l\in\NN,\qquad 2^{2^l}-l&\geq& 0\eq
we deduce
\bq
C&\leq &1+\sum_{l\in\lin 0, l_0-1\rin} \delta^{2^{2^l}-l}\\
&\leq & 1+l_0.\eq
It is not difficult to check that 
\bq
\fo l\geq 1,\qquad 2^{2^l}-l&\geq & \f12 2^{2^l}\eq
so that
\bq
l_0&=& \min\{l\in\N_+\st 2^{2^l}-l\geq 1/\log_2(1/\delta)\}\\
&\leq &\min\{l\in\N_+\st 2^{2^l}\geq 2/\log_2(1/\delta)\}\\
&=&1\vee\lceil \log_2\log_2(2/\log_2(1/\delta))\rceil.
\eq
The announced result follows from the fact
\bq
\log_2\log_2(2/\log_2(1/\delta))\geq 1 &\Leftrightarrow& \delta\geq \f1{\sqrt{2}}.\eq
\end{proof}

The following observations show that $Q$ needs to be at least decaying exponentially for the Hardy inequality approach to work. 
\begin{remark}\label{Rem17}
\par
a) In view of the expression of $\pi$, it is natural to try to replace \eqref{muH} by
\bq 
\fo x\in \N_+,\qquad \mu(x)&\df& 
 Z\pi(x)\\&=&
Q(x)Q(N(x)).\eq
But then in Lemma \ref{prel1}, where we want the ratios $\mu(T^*)/\nu(T)$ to be bounded above for singletons $T$, we end up with the fact that
\[
\frac{Q(N(m(T)))}{Q(p(m(T)))}=\frac{\mu(T)}{\nu(T)}\leq \frac{\mu(T^*)}{\nu(T)}
\]
must be bounded above for singletons $T$. Namely an extension of \eqref{QQQH} must hold: there exists a constant $c>0$ such that
\bqn{choi}
\fo x\in\N_+,\qquad Q(N(x))&\leq & c Q(p(x)).\eqn
Writing $x=y 2^p$, with $y\in\cI$ and $p\in\N$, we must have
\bq
Q(N(y2^p))&\leq & c Q(p).\eq
Take $y=1+2+4+\cdots +2^l$, then we get that
$p, p+1, ..., p+l$ all belong to $Q(N(y2^p))$, so that
\bq
Q(\{p, p+1, ..., p+l\})&\leq & c Q(p),\eq
and letting $l$ go to infinity, it follows that
\bq
Q(\lin p, \iy\lin)&\leq & c Q(p),\eq
namely, $Q$ has exponential tails.

b) Other subtrees of the graph generated by $K$ could have been considered. It amounts to choose the parent of any $x\in\N_+$.
But among all possible choices of such a neighbor, the one with most weight is $p(x)$, at least if $Q$ is decreasing.
In view of the requirement \eqref{choi}, it looks like the best possible choice.\par
c) If one is only interested in Proposition \ref{pro12} with $\mu$ defined by \eqref{muH}, then many more probability measures $Q$ can be considered, in particular any
polynomial probability of the form
\bq
\fo x\in\N,\qquad Q(x)&\df& \f1{\zeta(l)(x+1)^l}\eq
where $\zeta$ is the Riemann function and $l>1$.
\end{remark}


\appendix

\section{Dirichlet--Cheeger inequalities}\label{a1}

We begin by showing the Dirichlet--Cheeger inequality that  we have been using in the previous sections. 
It is a direct extension (even simplification) of the proof of the Cheeger inequality given in Saloff-Coste \cite{MR1490046}.
We end this appendix by proving that it is in general not possible to compare linearly the Dirichlet--Cheeger constant of an absorbed Markov chain with the largest Dirichlet--Cheeger constant induced on a spanning subtree.

Let us work in continuous time. Consider $L$ a sub-Markovian generator on a finite set $V$. Namely, $L\df (L(x,y))_{x,y\in V}$, whose off-diagonal entries are non-negative and whose row sums are non-positive.
Assume that $L$ is irreducible and reversible with respect to a probability $\pi$ on $V$.
\par
Let $\lambda(L)$ be the smallest eigenvalue of $-L$ (often called the Dirichlet eigenvalue). The variational formula for eigenvalues shows that
\bqn{lambda}
\lambda(L)&=&\min_{f\in \RR^V\setminus\{0\}}\f{-\pi[fL[f]]}{\pi[f^2]}.\eqn
The Dirichlet--Cheeger constant $\iota(L)$ is defined similarly, except that only indicator functions are considered in the minimum:
\bqn{iota}
\iota(L)&=&\min_{A\subset V,\, A\neq\emptyset}\f{-\pi[\un_AL[\un_A]]}{\pi[A]}.\eqn
Here is the Dirichlet--Cheeger inequality:
\begin{thm}\label{theo18}
Assuming $L\neq 0$, we have
\bq
\f{\iota(L)^2}{2\ell(L)}\ \leq \ \lambda(L)\ \leq \ \iota(L)\eq
where $\ell(L)\df \max\{\vert L(x,x)\vert\st x\in V\}>0$. 
\end{thm}
\par
When $L$ is Markovian, the above inequalities are trivial and reduce to $\iota(L)=\lambda(L)=0$. Indeed, it is sufficient to consider $f=\un$ and $A=V$ respectively in the r.h.s.\ of \eqref{lambda} and \eqref{iota}.
Thus there is no harm in supposing furthermore that $L$ is strictly sub-Markovian: at least one of the row sums is negative.
To bring this situation back to a Markovian setting, it is usual to extend $V$ into $\bar V\df V\sqcup\{0\}$ where $0\not\in V$ is a new point.
Then one introduces the extended Markov generator $\bar L$ on $\bar V$ via
\bq
\fo x, y\in \bar V,\qquad \bar L(x,y)&\df& \lt\{\begin{array}{ll}
L(x,y)&\hbox{if $x,y\in V$,}\\
-\sum_{z\in V}L(x,z)&\hbox{if $y=0$,}\\
0&\hbox{otherwise.}
\end{array}\rt.\eq
Note that the point $0$ is absorbing for the Markov processes associated to $\bar L$.

It is convenient to give another expression for $\iota(L)$.
Consider the set of edges 
\bq
\bar E&\df& \{\{x,y\}\st x\neq y \in \bar V\}.\eq
We define a measure $\mu$ on $\bar E$:
\bq
\fo e\df\{x,y\}\in \bar E,\qquad \mu(e)&\df& \lt\{\begin{array}{ll}
\pi(x)L(x,y)&\hbox{if $x,y\in V$,}\\
\pi(x) \bar L(x,0)&\hbox{if $y=0$,}
\\
\pi(y) \bar L(y,0)&\hbox{if $x=0$.}
\end{array}\rt.\eq
(Note that the reversibility assumption was used to ensure that the first line is well-defined.) Extend any $f\in\RR^V$ into the function $\bar f$ on $\bar V$ by making it vanish at 0 and define
\bq
\fo e\df\{x,y\}\in\bar E,\qquad \vert d\bar f\vert(e)&\df&
\vert \bar f(y)-\bar f(x)\vert.
\eq
With these definitions we can check that
\bq
\fo f\in\RR^V,\qquad -\pi[fL[f]]&=&\sum_{e\in\bar E}\vert d\bar f\vert^2(e) \mu(e).\eq
These notations enable to see \eqref{iota} as a $L^1$ version of \eqref{lambda}:
\begin{prop}\label{pro19} We have
\bq
\iota(L)&=&\min_{f\in\RR^V\setminus\{0\}}\f{\sum_{e\in\bar E}\vert d\bar f\vert(e) \mu(e)}{\pi[\vert f\vert]}.\eq
\end{prop}
\begin{proof}
Restricting the minimum in the r.h.s.\ to indicator functions, we recover the r.h.s.\ of \eqref{iota}. It is thus sufficient to show that for any given
$f\in\RR^V\setminus\{0\}$,
\bqn{L1}
\f{\sum_{e\in\bar E}\vert d\bar f\vert(e) \mu(e)}{\pi[\vert f\vert]}&\geq &\iota(L).\eqn
Note that 
\bq
\fo e\in\bar E,\qquad \vert d\bar f \vert(e)&\geq &  \vert d \vert \bar f\vert \vert(e),\eq
so without lost of generality, we can assume that $f\geq 0$. For any $t\geq 0$, consider the set $F_t$ and its indicator function given by
\bq
F_t&\df& \{\bar f>t\}\ =\ \{f>t\},\\
f_t&\df&\un_{F_t}.\eq
Note that
\bq
\fo x\in V,\qquad f(x)&=&\int_0^{+\iy} f_t(x)\, dt,\eq
so that  by integration,
\bq
\pi[f]&=&\int_0^{+\iy}\pi[F_t]\, dt.\eq
\par
Furthermore, we have
\bq
\sum_{e\in\bar E}\vert d\bar f\vert(e) \mu(e)&=&\sum_{e\fd\{x,y\}\st \bar f(y)>\bar f(x)}(\bar f(y)-\bar f(x))\mu(e)\\
&=&\sum_{e\fd\{x,y\}\st \bar f(y)>\bar f(x)}\int_{\bar f(x)}^{\bar f(y)}\mu(e)\, dt\\
&=&\int_{0}^{+\iy}\sum_{e\fd\{x,y\}\st \bar f(y)>t\geq \bar f(x)}\mu(e)\, dt\\
&=&
\int_{0}^{+\iy}\mu(\pa F_t)\, dt,
\eq
where for any $A\subset V$, we define
\bq
\pa A&\df& \{\{x,y\}\in \bar E\st x\in A\hbox{ and } y\not\in A\}.\eq
Note that for any such $A$, we have
\bq
\mu(\pa A)&=&-\pi[\un_A L[\un_A]],\eq
so that
\bq
\sum_{e\in\bar E}\vert d\bar f\vert(e) \mu(e)&=&-\int_{0}^{+\iy}\pi[f_tL[f_t]]\, dt\\
&\geq & \iota(L) \int_0^{+\iy}\pi[F_t]\, dt\\
&=&\iota(L) \pi[f],\eq
showing \eqref{L1}.
\end{proof}
We are now ready to prove Theorem \ref{theo18}.
\begin{proof}[Proof of Theorem \ref{theo18}]
Given $g\in \RR^V$, let $f=g^2$. By Proposition \ref{pro19}, we compute
\bq
\iota(L)\pi[ f]&\leq & \sum_{e\in\bar E}\vert d\bar f\vert(e) \mu(e)\\
&=&\sum_{e\fd\{x,y\}\in \bar E}\vert \bar g^2(y)-\bar g^2(x)\vert \mu(e)\\
&=&\sum_{e\fd\{x,y\}\in \bar E}\vert \bar g(y)-\bar g(x)\vert \vert \bar g(y)+\bar g(x)\vert \mu(e)\\
&\leq &\sqrt{\sum_{e\fd\{x,y\}\in \bar E}(\bar g(y)-\bar g(x))^2\mu(e)}\sqrt{\sum_{e\fd\{x,y\}\in \bar E}( \bar g(y)+\bar g(x) )^2\mu(e)}\\
&\leq & \sqrt{-\pi[gL[g]]}\sqrt{2\sum_{e\fd\{x,y\}\in \bar E}( \bar g^2(y)+\bar g^2(x) )\mu(e)}\\
&=&\sqrt{-\pi[gL[g]]}\sqrt{4\sum_{e\fd\{x,y\}\in \bar E}\bar g^2(x) \mu(e)}\\
&=&\sqrt{-\pi[gL[g]]}\sqrt{2\sum_{x\in V} g^2(x)\pi(x)\sum_{y\in \bar V\setminus\{x\}}\bar L(x,y) }\\
&=&\sqrt{-\pi[gL[g]]}\sqrt{2\sum_{x\in V} g^2(x)\pi(x)\vert L(x,x)\vert }\\
&\leq &\sqrt{2\ell(L)}\sqrt{-\pi[gL[g]]}\sqrt{\pi[g^2] }\\
&=&\sqrt{2\ell(L)}\sqrt{-\pi[gL[g]]}\sqrt{\pi[f]}.\eq
Thus, we have 
\bq
\f{\iota(L)^2}{2\ell(L)}\pi[g^2]&\leq & -\pi[gL[g]],\eq
which gives the desired lower bound for $\lambda(L)$. The upper bound is immediate.
\end{proof}

The unoriented graph associated to $L$ is $\bar G\df (\bar V,\bar E_L)$ where 
\bq
\bar E_L&\df& \{ e\in\bar E\st \mu(e)>0\}.\eq
Consider $\TT$, the set of all subtrees of $\bar G$, and for any $T\in\TT$, consider the sub-Markovian generator $L_T$ on $V$ associated to $T$ via\bq
L_T(x,y)&\df&\lt\{\begin{array}{ll}
L(x,y)&\hbox{if $\{x,y\}\in \bar E(T)$,}\\
-\sum_{z\in V\setminus\{x\}}L_T(x,z)&\hbox{if $x=y$ and $\{x,0\}\not\in \bar E(T)$,}\\
-\sum_{z\in V\setminus\{x\}}L_T(x,z)-\bar L(x,0)&\hbox{if $x=y$ and $\{x,0\}\in \bar E(T)$,}\\
0&\hbox{otherwise,}
\end{array}\rt.\eq
where $x,y\in V$ and $\bar E(T)$ is the set of (unoriented) edges of $T$.\par
Note that $L_T$ is also reversible with respect to $\pi$ (it is  irreducible if and only if $0$ belongs to a unique edge of $\bar E(T)$). Denote $\mu_T$ the corresponding measure on $\bar E$.
It is clear that $\mu_T\leq \mu$, so we get
\bq
\iota(L_T)&\leq & \iota(L).\eq
In the spirit of {Benjamini} and  {Schramm} \cite{zbMATH01054104}, we may wonder if conversely, $\iota(L)$ could be bounded above in terms of $\max_{T\in\TT}\iota(L_T)$. A linear comparison is not possible:
\begin{prop}\label{pro20}
It does not exist a universal constant $\chi>0$ such that for any $L$ as above,
\bq
 \chi \iota(L)&\leq &\max_{T\in\TT}\iota(L_T)\eq
\end{prop}
\begin{proof}
Let us construct a family $(\Ln)_{n\in\N_+}$ of sub-Markovian generators such that
\bqn{lim0}
\lim_{n\ri\iy}\f{\max_{T\in\TT}\iota(\Ln_T)}{\iota(\Ln)}&=&0\eqn
For any $n\in\N_+$, the state space $V^{(n)}$ of $\Ln$ is $\lin n\rin\times \{0,1\}$ (more generally, all notions associated to $\Ln$ will marked by the exponent $(n)$).
Denote $V_0^{(n)}\df \lin n\rin\times\{0\}$ and $V_1^{(n)}\df \lin n\rin\times\{1\}$.
We take
\bq
\Ln(x,y)&\df&
\lt\{\begin{array}{ll}
\epsilon &\hbox{if $x\in \Vn_i,\, y\in \Vn_{1-i}$ with $i\in\{0,1\}$,}\\
n\epsilon+1&\hbox{if $x=y\in \Vn_0$,}\\
n\epsilon&\hbox{if $x=y\in \Vn_1$,}\\
0&\hbox{otherwise,}\end{array}\rt.\eq
where $x,y \in V^{(n)}$, and $\epsilon>0$, that will depend on $n$, is such that $n\epsilon <1/2$.

Recall that $0$ is the cemetery point added to $V^{(n)}$, we have
\bq
\fo x\in\Vn,\qquad \bar L^{(n)}(x,0)&=&\lt\{\begin{array}{ll}
1&\hbox{if $x\in \Vn_0$,}\\
0&\hbox{if $x\in \Vn_1$.}
\end{array}\rt.\eq
Note that $\pi^{(n)}$ is the uniform probability on $\Vn$. Let us show that
\bqn{iLn1}
\iota(\Ln)&=&n\epsilon.\eqn
Consider any $\emptyset\neq A\subset \Vn$, and decompose $A=A_0\sqcup A_1$, with $A_0\df A\cap \Vn_0$ and $A_1\df A\cap \Vn_1$. Denote $a_0\df \vert A_0\vert$ and $a_1\df \vert A_1\vert$.
We have
\bq
\pa A&=&\{\{x,y\}\st x\in A_0,\, y\in \Vn_1\setminus A_1\}\\
& & \qquad \sqcup \{\{x,y\}\st x\in \Vn_0\setminus A_0,\, y\in  A_1\}\sqcup\{\{x,0\}\st x\in A_0\},\eq
and thus
\bq
\mu^{(n)}(\pa A)&=&\f1{2n}(\epsilon (a_0(n-a_1)+a_1(n-a_0))+a_0).\eq
It follows that
\bq
\f{\mu^{(n)}(\pa A)}{\pi^{(n)}(A)}&=&n\epsilon +\f{a_0(1-2\epsilon a_1)}{a_0+a_1}.\eq
Taking into account that $1-2\epsilon a_1>0$, the r.h.s.\ is minimized with respect to $a_0\in\lin 0, n\rin$ when $a_0=0$
and we then get (independently of $a_1$),
\bq
\f{\mu^{(n)}(\pa A)}{\pi^{(n)}(A)}&=&n\epsilon.\eq
We deduce \eqref{iLn1}.

Consider any $T\in\TT^{(n)}$ and let us check that
\bqn{iLn2}
\iota(\Ln_{T})&\leq &\epsilon.\eqn
Observe there exists $x\in \Vn_1$ such that there is a unique $y\in \Vn_0$ with $\{x,y\}$ being an edge of $T$.
Indeed, put on the edges of $T$ the orientation toward the root $0$. Thus from any vertex $x\in V_1^{(n)}$ there is a unique exiting edge (but it is possible there are several incoming edges). Necessarily, there is a vertex in $\Vn_0$ whose edge exits to $0$. So there are at most $n-1$ vertices from $\Vn_0$ whose exit edge points toward $\Vn_1$. In particular, there is at least one vertex from $\Vn_1$ which is not pointed out by a vertex from $\Vn_0$. We can take $x$ to be this vertex from $\Vn_1$ and $y\in \Vn_0$ is the vertex pointed out by the oriented edge exiting from $x$.

Considering the singleton $\{x\}$, we get
\bq
\mu^{(n)}_T(\pa \{x\})&=&\mu_T(\{x,y\})\ =\ \f{\epsilon}{2n},\\
\pi^{(n)}(x)&=&\f1{2n}.\eq
implying \eqref{iLn2} (a little more work would prove that an equality holds there). As a consequence, we see that
\bq
\max_{T\in\TT^{(n)}}\iota(\Ln_T)&\leq &\epsilon.\eq
Taking for instance  $\epsilon \df1/(4n)$ to fulfill the condition $n\epsilon <1/2$,
we obtain
\bq
\f{\max_{T\in\TT^{(n)}}\iota(\Ln_T)}{\iota(\Ln)}&\leq &\f1{n}\eq
and \eqref{lim0} follows.
\end{proof}

\section{Hardy's inequalities}\label{a2}
Our goal here is to extend the validity of Hardy's inequalities on finite trees to general denumerable trees, without assumption of local finiteness. We begin by recalling the Hardy's inequalities on finite trees. Consider 
 $\cT=(\bar V, \bar E,0)$ a finite tree rooted in 0, whose vertex and (undirected) edge sets are $\bar V$ and $\bar E$.
 Denote $V\df \bar V\setminus \{0\}$, for each $x\in V$, the parent $p(x)$ of $x$ is the neighbor of $x$ in the direction of 0.
 The other neighbors of $x$ are called the children of $x$ and their set is written $C(x)$. For $x=0$, by convention $C(0)$ is the set of neighbors of 0.
Let be given two positive measures $\mu,\nu$ defined on $V$.

Consider $c(\mu,\nu)$ the best constant $c\geq 0$ in the inequality
\bqn{c}
\fo f\in \RR^V,\qquad
\mu[f^2]&\leq & c\sum_{x\in V} (f(p(x))-f(x))^2\nu(x)\eqn
 where $f$ was extended to 0 via $f(0)\df 0$.

According to \cite{MR1709277} (see also  Evans, Harris and Pick \cite{MR1345719}), $c(\mu,\nu)$ can be estimated up to a factor 16 via Hardy's inequalities for trees, see \eqref{H} below.
To describe them we need several notations.

Let $\TT$ the set of subsets $T$ of $V$ satisfying the following conditions
\begin{itemize}
\item $T$ is non-empty and connected (in $\cT$),
\item $T$ does not contain 0,
\item if $x\in T$ has a child in $T$, then all children of $x$ belong to $T$.
\end{itemize}
Note that any $T\in\TT$ admits a closest element to 0, call it $m(T)$, we have $m(T)\neq 0$.
When $T$ is not reduced to the singleton $\{m(T)\}$, the  connected components  of $T\setminus\{m(T)\}$  are indexed by the set of the children of $m(T)$, namely $C(m(T))$. For $y\in C(m(T))$, denote by $T_y$  the connected component of $T\setminus\{m(T)\}$ containing $y$. Note that $T_{y}\in \TT$.

We extend $\nu$ as a functional on $\TT$, via the iteration
\begin{itemize}
\item when $T$ is the singleton $\{m(T)\}$, we take $\nu(T)\df \nu(m(T))$,
\item when $T$ is not a singleton, decompose $T$ as $\{m(T)\}\sqcup \bigsqcup_{y\in C(m(T))}T_{y}$, then $\nu$ satisfies
\bqn{rec}
\f1{\nu(T)}&=&\f1{\nu(m(T))}+\f1{\sum_{y\in C(m(T))} \nu(T_{y})}.\eqn
\end{itemize}
For $x\in V$, let $S_x$ be the set of vertices $y\in V$ whose path to 0 pass through $x$.
For any $T\in \TT$ we associate the subset 
\bq
T^*&\df& (S_{m(T)}\setminus T)\sqcup L(T)\eq
where $L(T)$ is the set of leaves of $T$, namely the $x\in T$ having no children in $T$. Equivalently, $T^*$ is the set of all descendants of the leaves of $T$, themselves included.

Consider $\SS\subset \TT$, the set of $T\in\TT$ which are such that $m(T)$ is a child of $0$. Finally, define
\bqn{b}
b(\mu,\nu)&\df& \max_{T\in\SS} \f{\mu(T^*)}{\nu(T)}.\eqn
We are interested in this quantity because of the Hardy inequality:
\bqn{H}
b(\mu,\nu)\ \leq \ {c}(\mu,\nu)\ \leq\ 16\,b(\mu,\nu).\eqn
Our goal here is to extend this inequality to the situation where $V$ is denumerable and where $\mu$ and $\nu$ are two positive measures on $V$, with $\sum_{x\in V}\mu(x)<+\iy$. 
\begin{remark}\label{unichild}
Without lost of generality, we can assume 0 has only one child, because what happens on different $S_x$ and $S_y$, where both $x$ and $y$ are children of $0$, can be treated separately.
\end{remark}
More precisely, while $V$ is now (denumerable) infinite, we first assume that the height of $\cT\df(\bar V, \bar E,0)$ is finite (implying that $\cT$ cannot be locally finite).
Recall that the height $h(x)$ of a vertex $x\in \bar V$ is the smallest number of edges linking $x$ to $0$. 
The assumption that $\sup_{x\in \bar V} h(x)<+\iy$ has the advantage that the iteration \eqref{rec} enables us to compute $\nu$ on $\TT$, starting from the highest vertices from an element of $\TT$.
Then $b(\mu,\nu)$ is defined exactly as in \eqref{b}, except the maximum has to be replaced by a supremum.
\par
Extend $c(\mu,\nu)$ as the minimal constant $c\geq 0$ such that \eqref{c} is satisfied, with the possibility that $c(\mu,\nu)=+\iy$ when  there is no such $c$.
Note that in \eqref{c}, the space $\RR^V$ can be reduced and replaced by $\cB(V)$, the space of bounded mappings on $V$:
\begin{lmm}\label{lem21}
We have
\bq
c(\mu,\nu)&=&\sup_{f\in\cB(V)\setminus\{0\}}\f{\mu[f^2]}{ \sum_{x\in V} (f(p(x))-f(x))^2\nu(x)}.\eq
\end{lmm}
\begin{proof}
Denote $\wi c(\mu,\nu)$ the above r.h.s. A priori we have $c(\mu,\nu)\geq \wi c(\mu,\nu)$. To prove the reverse bound, consider any $f\in \RR^V$
and consider for $M>0$, 
\[
f_M\df (f\wedge M)\vee(-M).
\]
Note that 
\bq
\sum_{x\in V} (f_M(p(x))-f_M(x))^2\nu(x)&\leq &\sum_{x\in V} (f(p(x))-f(x))^2\nu(x).
\eq
(This a general property of Dirichlet forms and comes from the 1-Lipschitzianity of the mapping $\RR\ni r\mapsto (r\wedge M)\vee(-M)$.) Since $f_M\in\cB(V)$, we have
\bq
\mu[f_M^2]&\leq & \wi c(\mu,\nu)\sum_{x\in V} (f_M(p(x))-f_M(x))^2\nu(x)\\
&\leq &\wi c(\mu,\nu)\sum_{x\in V} (f(p(x))-f(x))^2\nu(x).\eq
Letting $M$ go to infinity, we get at the limit by monotonous convergence
\bq
\mu[f^2]
&\leq &\wi c(\mu,\nu)\sum_{x\in V} (f(p(x))-f(x))^2\nu(x).\eq
Since this is true for all $f\in \RR^V$, we deduce that $c(\mu,\nu)\leq \wi c(\mu,\nu)$.
\end{proof}

 Consider $(x_n)_{n\in\N_+}$ an exhaustive sequence of $\bar V$, with $x_0=0$ and such that for any 
$n\in\N_+$, $\bar V_n\df\{x_0, x_1, ..., x_n\}$ is connected. We denote $\cT_n$ the tree rooted on 0 induced by $\cT$ on $\bar V_n$ and as above,
$V_n\df \bar V_n\setminus\{0\}=\{x_1, ..., x_n\}$.
For any $n\in\NN_+$ and $x\in V_n$, introduce the set
\bq
R_n(x)&\df& \{x\}\bigsqcup_{y\in C(x)\setminus V_n}S_y.\eq
In words, this is the set of elements of $V$ whose path to 0 first enters $V_n$ at $x$.\par
From now on, we assume that $0$ has only one child, taking into account Remark \ref{unichild}. It follows that 
\bqn{partit}
V&=&\bigsqcup_{x\in V_n} R_n(x).\eqn
Let $\mu_n$ and $\nu_n$ be the measures defined on $V_n$ via
\bq
\fo x\in V_n,\qquad \lt\{\begin{array}{rcl}\mu_n(x)&\df& \mu(R_n(x)),\\
\nu_n(x)&\df&\nu(x).\end{array}\rt.\eq
The advantage of the $\mu_n$ and $\nu_n$ is that they brought us back to the finite situation while enabling to approximate $c(\mu,\nu)$:
\begin{prop}\label{pro22}
We have
\bq
\lim_{n\ri\iy} c(\mu_n,\nu_n)&=&c(\mu,\nu).\eq
\end{prop}
\begin{proof}
We first check that the limit exists. For $n\in\NN_+$, consider the sigma-field $\cF_n$
generated by the partition \eqref{partit}. 
To each $\cF_n$-measurable function $f$, associate the function $f_n$ defined on $V_n$ by
\bq
\fo x\in V_n,\qquad f_n(x)&\df& f(x).\eq
This function determines $f$, since
\bq
\fo x\in V_n,\,\fo y\in R_n(x),\qquad f(y)&=&f_n(x).\eq
Furthermore, we have:
\bq
\mu[f^2]&= &\mu_n[f_n^2]\\
 \sum_{x\in V} (f(p(x))-f(x))^2\nu(x)&=& \sum_{x\in V_n} (f_n(p(x))-f_n(x))^2\nu_n(x).
\eq
It follows that
\bq
c(\mu_n,\nu_n)&=&\sup_{f\in\cB(\cF_n)\setminus\{0\}}\f{\mu[f^2]}{ \sum_{x\in V} (f(p(x))-f(x))^2\nu(x)},\eq
where $\cB(\cF_n)$ is the set of $\cF_n$-measurable functions, which are necessarily bounded, i.e., belong to $\cB(V)$. Since for any $n\in\N_+$ we have $\cF_n\subset\cF_{n+1}$, we get that the sequence $(c(\mu_n,\nu_n))_{n\in\NN_+}$ is non-decreasing and, taking into account Lemma \ref{lem21}, that
\bq
\lim_{n\ri\iy} c(\mu_n,\nu_n)&\leq &c(\mu,\nu).\eq
To get the reverse bound, first assume that $c(\mu,\nu)<+\iy$.
For given $\epsilon>0$, find a function $f\in\cB(V)$ with
\bq
\f{\mu[f^2]}{ \sum_{x\in V} (f(p(x))-f(x))^2\nu(x)}&\geq & c(\mu,\nu)-\epsilon.\eq
Consider $\pi$ the normalization of $\mu$ into a probability measure and let $f_n$ be the conditional expectation of $f$ with respect to $\pi$ and to the sigma-field $\cF_n$.
Note that the $f_n$ are uniformly bounded by $\lVe f\rVe_{\iy}$.
Thus by the bounded martingale convergence theorem and since $\pi$ gives a positive weight to any point of $V$, we have 
\bq
\fo x\in V,\qquad \lim_{n\ri\iy}f_n(x)&=&f(x).\eq
From Fatou's lemma, we deduce\bq
\lefteqn{\liminf_{n\ri\iy}
\sum_{x\in V_n} (f_n(p(x))-f_n(x))^2\nu_n(x)}\\&=&\liminf_{n\ri\iy}
\sum_{x\in V} (f_n(p(x))-f_n(x))^2\un_{V_n}(x)\nu(x)\\
&\geq & \sum_{x\in V} \liminf_{n\ri\iy} [(f_n(p(x))-f_n(x))^2\un_{V_n}(x)]\,\nu(x)\\
&=& \sum_{x\in V} (f(p(x))-f(x))^2\nu(x).
\eq
By another application of the bounded martingale convergence theorem, we get
\bq
\lim_{n\ri\iy} \mu_n[f_n^2]&=&\lim_{n\ri\iy} \mu[f_n^2]\\
&=&\mu[f^2],\eq
so that
\bq
\limsup_{n\ri\iy}\f{\mu_n[f_n^2]}{ \sum_{x\in V} (f_n(p(x))-f_n(x))^2\nu(x)}&\geq & \f{\mu[f^2]}{ \sum_{x\in V} (f(p(x))-f(x))^2\nu(x)}.\eq
It follows that
\bq
\lim_{n\ri\iy} c(\mu_n,\nu_n)&\geq &c(\mu,\nu)-\epsilon,\eq
 and since $\epsilon>0$ can be chosen arbitrary small,
 \bq
 \lim_{n\ri\iy} c(\mu_n,\nu_n)&\geq &c(\mu,\nu).\eq
 It remains to deal with the case where $c(\mu,\nu)=+\iy$.
 Then for any $M>0$, we can find a function $f\in\cB(V)$ with
 \bq
\f{\mu[f^2]}{ \sum_{x\in V} (f(p(x))-f(x))^2\nu(x)}&\geq & M.\eq
By the above arguments, we end up with 
 \bq
\lim_{n\ri\iy} c(\mu_n,\nu_n)&\geq &M,\eq
and since $M$ can be arbitrary large,
\bq
\lim_{n\ri\iy} c(\mu_n,\nu_n)&=&+\iy\ =\ c(\mu,\nu).\eq
 \end{proof}
Our next goal is to show the same result holds for $b(\mu,\nu)$. We need some additional notations. 
The integer $n\in\NN_+$ being fixed, denote $\TT_n$ and $\SS_n$ the sets $\TT$ and $\SS$ associated to $\cT_n$.
The functional $\nu_n$ is extended to $\TT_n$ via the iteration \eqref{rec} understood  in $\cT_n$.
\par
To any $T\in\TT_n$, associate $ T_n$ the minimal element of $\TT$ containing $T$.
It is obtained in the following way: to any $x\in T$, if $x$ has a child in $T$, then add all the children of $x$ in $V$, and otherwise do not add any other points.
\begin{lmm}\label{compwin}
We have the comparisons
\bq
\nu_n(T)&\geq & \nu( T_n),\\
\mu_n(T^*)&\leq &\mu( T_n^*),\eq
where $T^*$ is understood in $\cT_n$ (and $T_n^*$ in $\cT$).
\end{lmm}
\begin{proof}
The first bound is proven by iteration on the height of $T\in\TT_n$.

$\bullet$ If this height is zero, then $T$ is a singleton and $ T_n$ is the same singleton, so that $\nu_n(T)=\nu( T_n)$.
\par
$\bullet$ If the height $h(T)$ of $T$ is at least equal to 1, decompose
\bq
T&=&\{m_n(T)\}\sqcup \bigsqcup_{y\in C_n(m_n(T))}T_{n,y}\eq
where $m_n(\cdot)$, $C_n(\cdot)$ and $T_{n,\cdot}$ are the notions corresponding to $m(\cdot)$, $C(\cdot)$ and $T_{\cdot}$ in $\cT_n$.

Note that $T$ and $ T_n$ have the same height and decompose
\bq
 T_n&=&\{m( T_n)\}\sqcup \bigsqcup_{z\in C(m( T_n))} T_{n,z}.\eq
On the one hand, we have $m( T_n)=m_n(T)$ and $C_n(m_n(T))\subset C(m_n(T))$
and on the other hand, we have  for any $y\in C_n(m_n(T))$, 
\bq
\nu_n(T_y)&\geq & \nu( (T_y)_n)\\
&=&\nu( T_{n,y})\eq
due to the iteration assumption and to the fact that the common height of $T_y$ and $ (T_y)_n$ is at most equal to $h(T)-1$. The equality $ (T_y)_n=T_{n,y}$ is due to the fact that $T_{n,y}$ is obtained by the same completion of $T_y$ as the one presented for $T$ just above the statement of Lemma \ref{compwin}, and thus coincides with $ (T_y)_n$.\par
It follows that
\bq
\f1{\nu_n(T)}&=&\f1{\nu_n(m_n(T))}+\f1{\sum_{y\in C_n(m_n(T))} \nu_n(T_{y})}\\
&=&\f1{\nu(m( T_n))}+\f1{\sum_{y\in  C_n(m_n(T))} \nu_n(T_{y})}\\
&\leq & \f1{\nu(m( T_n))}+\f1{\sum_{y\in  C_n(m_n(T))} \nu(T_{n,y})}\\
&\leq & \f1{\nu(m( T_n))}+\f1{\sum_{y\in  C(m( T_n))} \nu(T_{n,y})}\\
&=&\f1{\nu( T_n)},
\eq
establishing the wanted bound 
\bq
\nu_n(T)&\geq & \nu( T_n).\eq
We now come to the second bound of the above lemma. By definition, we have
\bq
T^*&=& \sqcup_{x\in L_n(T)}S_{n,y},\eq
where $L_n(T)$ is the set of leaves of $T$ in $\cT_n$ and $S_{n,y}$ is the subtree rooted in $y$ in $\cT_n$.

Note that $L_n(T)\subset L(T_n)$ and by definition of $\mu_n$, we have
\bq
\fo y \in L_n(T),\qquad \mu_n(S_{n,y})&=&\mu(S_y).\eq
It follows that 
\bq
\mu_n(T^*)&=&\sum_{x\in L_n(T)}\mu_n(S_{n,y})\\
&=&\sum_{x\in L_n(T)}\mu(S_y)\\
&\leq & \sum_{x\in L(T_n)}\mu(S_y)\\
&=&\mu(T_n^*).\eq
\end{proof}
Let $\wi \SS_n$ be the image of $\SS_n$ under the mapping $\SS_n\ni T\mapsto T_n\in\SS$.
Since $\SS_n\ni T\mapsto T_n\in\wi\SS_n$ is a bijection, we get from Lemma \ref{compwin},
\bq
 b(\mu_n,\nu_n)&\df&\max_{T\in\SS_n} \f{\mu_n(T^*)}{\nu_n(T)}\\
 &\leq &\max_{T_n\in\SS_n} \f{\mu(T_n^*)}{\nu(T_n)}\\
 &\leq & b(\mu,\nu),\eq
so that
\bqn{limsupbb}
\limsup_{n\ri\iy}b(\mu_n,\nu_n)&\leq & b(\mu,\nu).\eqn
Let us show more precisely:
 \begin{prop}\label{pro23}
 We have
\bq
\lim_{n\ri\iy} b(\mu_n,\nu_n)&=&b(\mu,\nu).\eq
\end{prop}
 \begin{proof}
 According to \eqref{limsupbb}, it remains to show that
 \bqn{liminfbb}
\liminf_{n\ri\iy}b(\mu_n,\nu_n)&\geq & b(\mu,\nu).\eqn
 Consider $T\in\SS$ such that the ration $\mu(T^*)/\nu(T)$ serves to approximate $b(\mu,\nu)$, namely up to an arbitrary small $\epsilon>0$ if $b(\mu,\nu)<+\iy$  or is an arbitrary large quantity if 
 $b(\mu,\nu)=+\iy$.
Define
\bq
\fo n\in\NN_+,\qquad T^{(n)}&\df& T\cap V_n.\eq
Arguing as at the end of the proof of Proposition
\ref{pro22}, we will deduce
\eqref{liminfbb}
from
\bq
\lim_{n\ri\iy} \f{\mu_n((T^{(n)})^*)}{\nu_n(T^{(n)})}&= & \f{\mu(T^*)}{\nu(T)},\eq
where $(T^{(n)})^*$ is understood in $\cT_n$. This convergence will itself be the consequence of
\bqn{check1}
\lim_{n\ri\iy} \mu_n((T^{(n)})^*)&= & \mu(T^*),\\
\label{check2}\lim_{n\ri\iy} {\nu_n(T^{(n)})}&= & \nu(T).\eqn
For \eqref{check1}, note that
\bq
\mu(T^*)&=&\sum_{x\in L(T)}\mu(S_y),\eq
and as we have seen at the end of the proof of Lemma \ref{compwin},
\bq
\mu(T^*)&=&\sum_{x\in L_n(T^{(n)})}\mu(S_y).\eq
Thus \eqref{check1} follows by dominated convergence (since $\mu(V)<+\iy$), from
\bq
\fo x\in T,\qquad \lim_{n\ri\iy} \un_{L_n(T^{(n)})}(x)&=&\un_{L(T)}(x).\eq
To show the latter convergences, consider two cases:

$\bullet$ If $x\in L(T)$, then we will have $x\in L_n(T^{(n)})$ as soon as $x\in V_n$.

$\bullet$ If $x\in T\setminus L(T)$, then we will have $x\not\in L_n(T^{(n)})$ as soon as 
$V_n$ contains one of the children of $x$ in $T$.

We now come to \eqref{check2}, and more generally let us prove by iteration over their height, that for any $\wi T\in\TT$ and $\wi T\subset T$, we have
\bqn{nunu}
\lim_{n\ri\iy}\uparrow {\nu_n(\wi T\cap V_n)}&= & \nu(\wi T),
\eqn
i.e.,  the limit is non-decreasing.

Indeed, if $\wi T$ has height 0, it is a singleton $\{x\}$, 
we have $\nu_n(\wi T\cap V_n)=\nu(\wi T)$ as soon as $x$ belongs to $V_n$, insuring \eqref{nunu}.

Assume that $\wi T$ has height a $h\geq 1$ and that \eqref{nunu} holds for any $\wi T$ whose height is at most equal to $h-1$.
Write as usual
\bqn{wiT1}
\f1{\nu(\wi T)}&=&\f1{\nu(m(\wi T))}+\f1{\sum_{y\in C(m(\wi T))} \nu(\wi T_{y})}.\eqn
Assume that $n$ is large enough so that $C(m(\wi T))\cap V_n\neq \emptyset$ and in particular $m(\wi T)\in V_n$ and $m_n(\wi T\cap V_n)=m(\wi T)$.
Thus we also have
\bqn{wiT2}
\nonumber\f1{\nu_n(\wi T\cap V_n)}&=&\f1{\nu_n(m_n(\wi T\cap V_n))}+\f1{\sum_{y\in C_n(m_n(\wi T\cap V_n))} \nu_n((\wi T\cap V_n)_{y})}\\
&=&\f1{\nu(m(\wi T))}+\f1{\sum_{y\in C_n(m(\wi T))} \nu_n(\wi T_y\cap V_n)}.\eqn
On the one hand, the set $C_n(m(\wi T))$ is non-decreasing and its limit is $C(m(\wi T))$, and on the other hand, due to the induction hypothesis, we have for any $y\in C(m(\wi T))$,
\bq
\lim_{n\ri\iy}\uparrow \nu_n(\wi T_y\cap V_n)&= & \nu(\wi T_y).
\eq
By monotone convergence, we get
\bq
\lim_{n\ri\iy}\uparrow \sum_{y\in C_n(m(\wi T))} \nu_n(\wi T_y\cap V_n)&=&\sum_{y\in C(m(\wi T))} \nu(\wi T_{y}),\eq
which leads to \eqref{nunu}, via \eqref{wiT1} and \eqref{wiT2}. This ends the proof of \eqref{liminfbb}.
\end{proof}
The conjunction of Propositions \ref{pro22} and \ref{pro23} leads to the validity of \eqref{H}, when $V$ is denumerable with $\cT$ of finite height.

Let us now remove the assumption of finite height. The arguments are very similar to the previous one, except that the definition of $b(\mu,\nu)$ has to be modified ($\mu$ and $\nu$ are still positive measures on $V$, with $\mu$ of finite total mass).\par
More precisely, for any $M\in\NN_+$, consider 
\bq
V_M&\df& \{x\in V\st h(x)\leq M\}.\eq
Define on $V_M$ the measure $\nu_M$ as the restriction to $V_M$ of $\nu$ and $\mu_M$ via
\bq
\fo x\in V_M,\qquad\mu_M(x)&\df& \lt\{\begin{array}{ll}
\mu(x)&\hbox{if $h(x)<M$,}\\
\mu(S_x)&\hbox{if $h(x)=M$.}\end{array}\rt.
\eq
 By definition, we take
 \bq
 b(\mu,\nu)&\df&\lim_{M\ri\iy} b(\mu_M,\nu_M).\eq
 This limit exists and the convergence is monotone, since he have
  \bq
\fo M\in\NN_+,\qquad  b(\mu_M,\nu_M)&=&\max_{T\in\SS_M} \f{\mu(T^*)}{\nu(T)}.\eq
where
  \bq
 \SS_M&\df& \{T\in\SS\st T\subset V_M\}.\eq
Note that a direct definition of $b(\mu,\nu)$ via the iteration \eqref{rec} is not possible: we could not start from leaves that are singletons.

 By definition, $c(\mu,\nu)$ is the best constant in \eqref{c}. It also satisfies
  \bq
c(\mu,\nu)&\df&\lim_{M\ri\iy} c(\mu_M,\nu_M),\eq
 as can be seen by adapting the proof of Proposition \ref{pro22}. We conclude 
 that \eqref{H} holds by passing at the limit in 
 \bq
 \fo M\in\NN_+,\qquad
 b(\mu_M,\nu_M)\ \leq \ {c}(\mu_M,\nu_M)\ \leq\ 16\,b(\mu_M,\nu_M).\eq


 \bibliographystyle{plain}

\end{document}